\documentclass[11pt]{amsart}
\usepackage{amsmath,amssymb,mathrsfs}
\usepackage[colorlinks=true,linkcolor=black,citecolor=black,urlcolor=black]{hyperref}

\newcommand{\define}{\textbf}

\newcommand{\excise}[1]{}

\newcommand{\isom}{\cong}
\renewcommand{\setminus}{\smallsetminus}
\renewcommand{\phi}{\varphi}

\renewcommand{\tilde}{\widetilde}

\renewcommand{\bar}{\overline}

\renewcommand{\AA}{\mathbb{A}}
\newcommand{\BB}{\mathbb{B}}
\newcommand{\CC}{\mathbb{C}}
\newcommand{\EE}{\mathbb{E}}
\newcommand{\PP}{\mathbb{P}}

\newcommand{\ZZ}{\mathbb{Z}}

\newcommand{\qq}{\mathbf{q}}   
\newcommand{\dd}{\mathbf{d}}   
\newcommand{\nn}{\mathbf{n}}   
\newcommand{\ee}{\mathbf{e}}

\newcommand{\ul}{\mathbf}

\newcommand{\pt}{\mathrm{pt}}  

\newcommand{\OO}{\mathcal{O}}  
\newcommand{\Aa}{\mathscr{A}}
\newcommand{\Bb}{\mathscr{B}}

\newcommand{\Fl}{{Fl}}

\newcommand{\FFl}{\mathbf{Fl}}  

\newcommand{\OOmega}{\mathbf{\Omega}} 
\newcommand{\D}{D}              
\newcommand{\DD}{\mathbf{D}}    

\newcommand{\mor}{M}            

\newcommand{\quot}{\mathscr{Q}} 
\newcommand{\hq}{\mathscr{Q}}   
\newcommand{\uu}{\mathscr{U}}

\newcommand{\Mbar}{\bar{M}}     

\newcommand{\Sch}{\mathfrak{S}} 
\newcommand{\scl}{\sigma}       

\newcommand{\qp}{\star}         
\newcommand{\qtp}{\circ}        

\newcommand{\ev}{\mathrm{ev}}   

\newcommand{\flag}{\Fl(\nn)}    
\newcommand{\Sn}{S^{\nn}}        

\newcommand{\qpi}{\tilde\pi}    

\newcommand{\Mm}{\mathbf{M}}			
\newcommand{\Mmbar}{\mathbf{\Mbar}}	
\newcommand{\Qq}{\mathcal{Q}}		
\newcommand{\Uu}{\mathcal{U}}
\newcommand{\Xx}{\mathbf{X}}

\DeclareMathOperator{\codim}{codim}

\DeclareMathOperator{\rk}{rk}
\DeclareMathOperator{\rank}{rank}
\DeclareMathOperator{\Span}{span}

\DeclareMathOperator{\Hom}{Hom}

\newtheorem{theorem}{Theorem}[section]
\newtheorem{lemma}[theorem]{Lemma}
\newtheorem{proposition}[theorem]{Proposition}
\newtheorem{corollary}[theorem]{Corollary}

\theoremstyle{definition}

\newtheorem{remark}[theorem]{Remark}
\newtheorem{example}[theorem]{Example}

\begin{document}

\title{Equivariant quantum Schubert polynomials}
\author{Dave Anderson}
\address{Department of Mathematics\\University of Washington\\Seattle, WA 98195}
\email{dandersn@math.washington.edu}
\author{Linda Chen}
\address{Department of Mathematics and Statistics\\Swarthmore College\\Swarthmore, PA 19081}
\email{lchen@swarthmore.edu}
\date{February 14, 2012}
\thanks{DA was partially supported by NSF Grant DMS-0902967. LC was partially supported by NSF Grant DMS-0908091 and NSF Grant DMS-1101625.}

\begin{abstract}
We establish an equivariant quantum Giambelli formula for partial flag varieties.  The answer is given in terms of a specialization of universal double Schubert polynomials.  Along the way, we give new proofs of the presentation of the equivariant quantum cohomology ring, as well as Graham-positivity of the structure constants in equivariant quantum Schubert calculus.
\end{abstract}

\maketitle

\section{Introduction}\label{s:intro}

Classical Schubert calculus is concerned with the cohomology rings of Grassmannians and (partial) flag varieties.  In recent years, equivariant and quantum versions of Schubert calculus have been developed.  Key ingredients in each of these theories are a presentations of the ring and ``Giambelli formulas'' expressing the additive basis of Schubert classes in terms of the presentation.  For example, the ordinary cohomology of the Grassmannian is generated by Chern classes, and the classical Giambelli formula in this context states that a Schubert class is represented by a Schur polynomial, which has a determinantal expression in terms of Chern classes.

The past fifteen years have seen great progress in modern Schubert calculus.  Many authors---including Bertram, Buch, Ciocan-Fontanine, Coskun, Kresch, Knutson, Mihalcea, Tamvakis, and Vakil---have proven results on quantum cohomology, equivariant cohomology, $K$-theory, and, more recently, equivariant $K$-theory and equivariant quantum cohomology.  The latter theory has connections with affine Schubert calculus: following ideas of Peterson, the relationship between (equivariant) homology of affine Grassmannians and (equivariant) quantum cohomology of partial flag varieties has been developed by Lapointe and Morse \cite{lm} and Lam and Shimozono \cite{ls-acta,ls}.  The associated combinatorics of $k$-Schur functions imparts further interest to the study of equivariant quantum Schubert calculus.

In this article, we study the \emph{equivariant quantum cohomology ring} of a partial flag variety; the main results give Giambelli formulas for Schubert classes.  Specifically, we define \emph{equivariant quantum Schubert polynomials} $\Sch^q_w(\sigma,t)$ (or more simply, $\Sch^q_w(x,t)$ in the case of  complete flag varieties) as specializations of Fulton's universal double Schubert polynomials.  This specialization is analogous to the specialization of universal (single) Schubert polynomials to Fomin-Gelfand-Postnikov's quantum Schubert polynomials for the quantum cohomology of complete flag varieties \cite{fgp}.  In the case of complete flags, our equivariant polynomial $\Sch^q_w(x,t)$ is equal to the specialization studied by Kirillov-Maeno under the name \emph{quantum double Schubert polynomial} \cite{kima}.  As a key feature of our point of view, we obtain a direct relationship between universal double Schubert polynomials, which solve a degeneracy locus problem, and their specialization to (equivariant) quantum Schubert polynomials (cf.~\cite{chen}).

Fix $\nn = (0<n_1<\cdots<n_m<n)$ and let $\Fl(\nn)$ denote the partial flag variety parametrizing flags $V_1 \subset \cdots \subset V_m \subset \CC^n$, with $\dim V_i = n-n_i$.  Recall that the Schubert classes $\scl_w$ are parametrized by certain permutations $\Sn \subseteq S_n$ (see \S\ref{ss:flags}).  We write $\Fl(n)$ for the complete flag variety, corresponding to the case $\nn = \{1,\ldots,n-1\}$.

To state the main theorem, we describe a specialization of the \emph{universal double Schubert polynomial} $\Sch_w(g,h)$, which is a polynomial in variables $g_i[j]$ and $h_i[j]$ defined by Fulton \cite{fulton}.  Specialize these variables as follows:
\begin{equation}\label{e:spec1}
\begin{aligned}
  g_i[0] & \mapsto x_i  & &\text{for }1\leq i\leq n, \\
  g_{n_{i-1}+1}[n_{i+1}-n_{i-1}-1] & \mapsto (-1)^{n_i-n_{i-1}+1} q_i & & \text{for }1\leq i\leq m, \\
  h_i[0] & \mapsto t_i  & &\text{for }1\leq i\leq n,  \text{ and} \\
  g_i[j],\; h_i[j] & \mapsto 0 & & \text{for all other } i,j.
\end{aligned}
\end{equation}
Writing $\sigma_i^j$ for the $i$th elementary symmetric polynomial in the variables $x_{n_{j-1}+1},\ldots,x_{n_j}$, the equivariant quantum Schubert polynomial $\Sch_w^q(\sigma,t)$ is the result of these specializations applied to $\Sch_w(g,h)$.  (See \S\ref{s:univ-sch} for more background on universal Schubert polynomials and these specializations.)

\begin{theorem}
\label{t:main}
For any permutation $w\in \Sn$, we have
\[
  \scl_w = \Sch^q_w(\sigma,t)
\]
as classes in $QH_T^*(Fl(\nn))$.
\end{theorem}
\noindent

\noindent
As an example, the equivariant quantum Schubert polynomials for $\Fl(3)$ are listed in Table~\ref{table3}.

\begin{table}[h]
\[
\begin{array}{|l|l|} \hline
 w    &  \Sch^q_w(x,t)    \\ \hline\hline
123   &  1  \\ \hline
213   & x_1 - t_1 \\ \hline
132   & x_1 + x_2 - t_1 - t_2   \\ \hline
231   & x_1\,x_2 + q_1 - \left( x_1 + x_2 \right) t_1 + t_1^{2} \\ \hline
312   & x_1^{2} - q_1 - x_1\,(t_1 + t_2) + t_1\,t_2 \\ \hline
321   & \left( x_1 - t_2 \right)  \left( x_1\,x_2 + q_1 - \left( x_1 + x_2 \right) t_1 + t_1^{2} \right)  \\ \hline
\end{array}
\]
\caption{Equivariant quantum Schubert polynomials for $Fl(3)$. \label{table3}}
\end{table}

Our approach is to adapt the methods in \cite{chen} to the equivariant situation.  One of the main geometric tools used in previous approaches to quantum Giambelli formulas (cf.~\cite{bertram,cf-partial,chen} is a moving lemma for quot schemes which relies on general position arguments not immediately available in equivariant cohomology.  To surmount this technical difficulty, we use the ``mixing group'' action introduced in \cite{anderson} to prove an equivariant moving lemma.  
The new equivariant moving lemma proved in \S\ref{s:moving} should be of independent interest; we use a version to prove positivity of equivariant Gromov-Witten invariants in \cite{ac}.

For the present paper, the main consequence of the moving lemmas is that the equivariant quantum structure constants can be computed on quot schemes (Proposition~\ref{p:product-quot}).  This allows us to use the inductive method of \cite{chen} to prove the equivariant quantum Giambelli formula (\S\ref{s:proofs}).  We expect it to have further applications, since it also allows one to use the apparatus of equivariant localization on quot schemes, which is well understood, thanks to \cite{bcs}.

Our results immediately recover the presentation of the equivariant quantum ring presentation computed in \cite{kim-eq} (see Corollary~\ref{c:presentation}), the equivariant quantum Giambelli formula for Grassmannians in terms of factorial Schur polynomials in \cite{mihalcea-giambelli}, and the equivariant Giambelli formula for flag varieties in terms of double Schubert polynomials in \cite{knmi}.  
In addition, our approach recovers the Graham-positivity result in \cite{mihalcea-positivity}: as polynomials in a natural choice of variables, the structure constants for equivariant quantum multiplication have nonnegative coefficients (Corollary~\ref{c:eqlr-enum}).

Like their non-equivariant counterparts, the equivariant quantum Schubert polynomials possess a stability property: the same polynomial represents a Schubert class $\sigma_w$ of codimension $\ell(w)$, independently of which flag variety $\Fl(n)$ it is in, for sufficiently large $n$.  In \S\ref{s:stability}, we give a precise statement to this effect.  A useful computational consequence is that equivariant quantum products are computed (on the nose, not up to an ideal) by products of Schubert polynomials, at least for sufficiently large $n$.  We include some examples, using the equivariant quantum Schubert polynomials to produce multiplication tables for $QH_T^*\Fl(n)$ (for small $n$).

After the results presented here were announced, the equivariant quantum Giambelli formula was proved by Lam and Shimozono, using different methods \cite{ls2}.

\smallskip
\noindent
{\it Acknowledgements.}  This project began in March 2010 at the AIM workshop on Localization Techniques in Equivariant Cohomology, and we thank William Fulton, Rebecca Goldin, and Julianna Tymoczko for organizing that meeting.  We also thank Sara Billey for helpful comments, and Anders Buch for sharing Maple code which we modified to compute equivariant quantum products.  DA is grateful for the hospitality of the Mathematics Department at the University of British Columbia, where much of this work took place.

\section{Background and notation}\label{s:back}

%
\subsection{Flag varieties}\label{ss:flags}

We recall some basic facts about partial flag varieties.  Let
\[
  \nn = \{0=n_0<n_1<\cdots<n_m<n_{m+1}=n\}
\]
be a strictly increasing sequence of integers, and let $V=\CC^n$.  The \emph{$m$-step partial flag variety} $\Fl(\nn)=\Fl(n_1,\ldots,n_m;V)$ parametrizes flags $V_\bullet = (V_m \subset \cdots \subset V_1 \subset V)$, with $\dim V_i = n-n_i$; equivalently, $\Fl(\nn)$ parametrizes successive quotients $V/V_i$ of dimension $n_i$.  This is a smooth projective variety of dimension $\dim\Fl(\nn)=\sum_{i=1}^m n_i(n_{i+1}-n_i)$, and it comes equipped with a universal sequence of quotient bundles:
%
%
\[
 V_{\Fl(\nn)}=Q_{m+1} \twoheadrightarrow \cdots \twoheadrightarrow Q_2 \twoheadrightarrow Q_1,
\]
where $Q_i$ is the vector bundle of rank $n_i$ whose fiber  over $V_\bullet$ is $V/V_{i}$. When $\nn=\{1,2,\ldots,n-1\}$, we obtain the complete flag variety, which we write as $Fl(\CC^n)$.

These bundles give generators for cohomology ring of $\Fl(\nn)$, as follows.  For $1\leq j\leq m+1$, let $x_{n_{j-1}+1},\ldots,x_{n_j}$ be the Chern roots of the bundle $\ker(Q_j\rightarrow Q_{j-1})$.  For $1\leq i \leq n_j-n_{j-1}$, set
\begin{equation}\label{e:sigma-def1}
  \sigma_i^j=c_i(\ker(Q_j\rightarrow Q_{j-1}));
\end{equation}
this is the $i$th elementary symmetric polynomial in $x_{n_{j-1}+1},\ldots,x_{n_j}$.  
Since the Chern class $c_k(Q_j)$ is symmetric in $x_{n_{j-1}+1},\ldots,x_{n_j}$ for every $1\leq j\leq l$, it can be written as a polynomial in $\sigma_i^j$, which we denote by $\tilde{e}_k(l)(\sigma)$ or $\tilde{e}_k(l)$.  
Note that when $\Fl(\nn)$ is the complete flag variety $\Fl(\CC^n)$, $\sigma_i^j$ is defined for $i=1$, and in that case, $\sigma_1^j=x_j$.

\begin{theorem}\label{t:pres}
The cohomology ring of $\Fl(\nn)$ is presented as 
\[
 H^*(\Fl(\nn)) \isom \ZZ[\sigma_1^1,\ldots,\sigma_{n_1}^1,\ldots,\sigma_{1}^{{m+1}},\ldots,\sigma_{n-n_m}^{{m+1}}]/I,
\]
where $I$ is the ideal $(\tilde{e}_1(m+1),\ldots,\tilde{e}_n(m+1))$.
\end{theorem}

In the special case of the complete flag variety $\Fl(\CC^n)$, this gives
\[
 H^*(\Fl(\CC^n)) \isom \ZZ[x_1,\ldots,x_n]/(e_1(x),\ldots,e_n(x))
\]
where $x_i=c_1(\ker(Q_j\rightarrow Q_{j-1}))$ and $e_i(x)$ is the $i$th elementary symmetric polynomial in $x_1,\ldots,x_n$ for  $1\leq i\leq n$ .

\subsection{Permutations and Bruhat order}\label{ss:perms}

The cohomology ring $H^*(\Fl(\nn))$ has a $\ZZ$-basis of Schubert classes, indexed by permutations in the set
\[
  \Sn:=\{ w\in S_n: w(i)<w(i+1) \text{ if } n-i\not\in\nn\}.
\]
If $S_\nn$ denotes the subgroup of $S_n$ generated by adjacent transpositions $(i,i+1)$ for $n-i\not\in\nn$, then $\Sn$ is a set of coset representatives for $S_n/S_\nn$.  (For the complete flag variety $\Fl(\CC^n)$, so $\nn = \{1,2,\ldots,n-1\}$, we have $\Sn = S_n$.)

For any permutation $w\in S_n$, define
\[
  r_w(p,q) = \#\{i \leq p \,|\, w(i)\leq q \}.
\]
This is the rank of the upper-left $p\times q$ submatrix of the permutation matrix corresponding to $w$ (which has $1$'s in positions $(i,w(i))$ and $0$'s elsewhere).

The \emph{Bruhat order} on $S_n$ is a partial order which may be defined by $v \leq w$ iff $r_v(p,q)\geq r_w(p,q)$ for all $1\leq p,q\leq n$.  The \emph{length} of $w$ is the number
\[
  \ell(w) = \#\{ i<j \,|\, w(i)>w(j) \}.
\]
The Bruhat order is ranked by length: $v\leq w$ implies $\ell(v)\leq\ell(w)$.  There is a unique permutation of greatest length, denoted $w_\circ$; it is given by $w_\circ(i) = n+1-i$.

The subset $\Sn\subseteq S_n$ is also characterized as the set of minimal-length coset representatives for $S_n/S_\nn$.  Given any permutation $w\in S_n$, one gets an element $\bar{w}$ of $\Sn$ by taking the representative of $w S_\nn$ of smallest length; concretely, this means sorting the entries of each block $[w(n_{i}+1),\,w(n_{i}+2),\ldots,\,w(n_{i+1})]$ into increasing order.

Bruhat order induces a partial order on the subset $\Sn$, and the unique permutation of greatest length in $\Sn$ is $w^\circ := \bar{w_\circ}$, given explicitly by 
\[
  w^\circ = [ n_m+1,n_m+2,\ldots,n,\; n_{m-1}+1, n_{m-1}+2, \ldots,n_m, \ldots,n_2,\; 1,2,\ldots,n_1].
\]
Its length is $\ell(w^\circ) = \dim\Fl(\nn) = \sum_{i=1}^m n_i(n_{i+1}-n_i)$.

\subsection{Schubert cells and Schubert varieties}

Fix a basis $\{e_1,\ldots,e_n\}$ for $V=\CC^n$.  The \emph{standard flag} $E_\bullet$ is defined by $E_i = \Span\{e_1,\ldots,e_i\}$ and the \emph{opposite flag} $\tilde{E}_\bullet$ is defined by $\tilde{E}_i = \Span\{e_n,e_{n-1},\ldots,e_{n-i+1}\}$.  
The \emph{Schubert varieties} in $\Fl(\nn)$ can be described in several ways. Fixing the standard flag $E_\bullet$, we define $\Omega_w = \Omega_w(E_\bullet) \subseteq \Fl(\nn)$ by
\[
  \Omega_w = \{ V_\bullet \in \Fl(\nn) \,|\, \rk(E_q \to \CC^n/V_p) \leq r_w(n_p,q) \text{ for all } 1\leq q\leq n, n_p\in\nn\}
\]
This is the closure of the \emph{Schubert cell}
\[
  \Omega_w^\circ = \{ V_\bullet \in \Fl(\nn) \,|\, \rk(E_q \to \CC^n/V_p) = r_w(n_p,q) \text{ for all } 1\leq q\leq n, n_p\in\nn\},
  \]
which is isomorphic to $\AA^{\dim\Fl(\nn)-\ell(w)}$.  Replacing $E_\bullet$ with the opposite flag $\tilde{E}_\bullet$, we obtain the \emph{opposite Schubert varieties} $\tilde\Omega_w$ and the \emph{opposite Schubert cells} $\tilde\Omega_w^\circ$. Here, the Bruhat order on $\Sn$ is identified with the order induced by inclusions of Schubert varieties: $v\leq w$ iff $\Omega_v \supseteq \Omega_w$ iff $\tilde{\Omega}_v\subseteq\tilde\Omega_w$.

Regarding $E_\bullet$ as a flag of trivial vector bundles on $\Fl(\nn)$, the Schubert variety may be defined equivalently as the degeneracy locus of points $x\in\Fl(\nn)$ where $\rk_x(E_q \to Q_p)\leq r_w(n_p,q)$, and similarly for the opposite Schubert varieties.

The Schubert cells give an (affine) cell decomposition of the partial flag variety, so the classes of Schubert varieties form a linear basis for the cohomology ring of $\Fl(\nn)$:
\[
  H^{*}\Fl(\nn) = \bigoplus_{w\in \Sn} \ZZ\cdot [\Omega_w].
\]
Since $\Omega_w$ has codimension $\ell(w)$, its class lies in $H^{2\ell(w)}\Fl(\nn)$.  

The Schubert classes $[\Omega_w]$ are written in terms of this presentation as \emph{Schubert polynomials} $\Sch_w(x)$ \cite{bgg}\cite{d}\cite{lsch}.  For the moment, consider the complete flag variety $\Fl(n)$.  Let $w_\circ$ be the longest permutation in $S_n$, and write $w=w_\circ s_{i_1}\ldots s_{i_k}$, where  $s_i$ is the simple transposition $(i, i+1)$ and $k= \binom{n}{2}-\ell(w)$.  For $1\leq i< n$, let $\partial_i$ be the divided difference operator acting on $\ZZ[x_1,\ldots,x_n]$ by
\begin{equation}\label{e:divdiff}
 \partial_i P = \frac{P(x_1,\ldots,x_n)-P(x_1,\ldots,x_{i-1},x_{i+1},x_i,x_{i+2},\ldots,x_n)}{x_i-x_{i+1}}.
\end{equation}

Then the Schubert polynomials are defined by 
\[
  \Sch_w(x)=\partial_{i_k}\cdots \partial_{i_1} (x_1^{n-1}x_2^{n-2}\ldots x_{n-1}).
\]
The \emph{Giambelli formula} gives $[\Omega_w]$ in terms of the presentation.  For the complete flag variety $\Fl(\CC^n)$, we have for any $w\in S_n$,
\[
  [\Omega_w]=\Sch_w(x) \text{ as classes in } H^*(\Fl(\CC^n)).
\]

For $w\in\Sn$, $\Sch_w(x)$ can be written as a polynomial in the classes $\sigma_i^j$ from Theorem~\ref{t:pres}.  We write this polynomial as $\Sch_w^\nn(\sigma)$, and for partial flag varieties, we have

\begin{theorem}
For any $w\in\Sn$, $[\Omega_w]=\Sch_w^\nn(\sigma)$ as classes in $H^*(\Fl(\nn))$.
\end{theorem}

The Schubert cells may also be described as orbits, and we will need this point of view.  The group $GL_n$ acts on $\Fl(\nn)$ via its action on $\CC^n$, and the parabolic subgroup $P$ of staircase (block upper-triangular) matrices fixes the standard flag.  This gives rise to an isomorphism $\Fl(\nn) \isom GL_n/P$ with a decomposition into Schubert cells indexed by permutations $w\in\Sn$.  Given $w\in \Sn$, let $p(w) \in \Fl(\nn)$ be the flag $(F_m \subset \cdots \subset F_1 \subset V)$ with $F_i = \Span\{ e_{w(n)}, e_{w(n-1)}, \ldots, e_{w(n_i+1)} \}$.  We have
\begin{align*}
 \Omega_w^\circ &= B\cdot p(w) \\
\intertext{and}
 \tilde\Omega_w^\circ &= \tilde{B}\cdot p(\bar{w_\circ w}) ,
\end{align*}
where $B$ and $\tilde{B}$ are the groups of upper- and lower-triangular matrices, respectively.  
The cells $\Omega_w$ and $\tilde\Omega_{\bar{w_\circ w}}$ intersect transversally in the point $p(w)$, and it follows from this (together with dimension considerations) that the classes $[\tilde\Omega_{\bar{w_\circ w}}]$ form a Poincar\'e dual basis: writing $\pi:\Fl(\nn)\to\pt$,
\[
   \pi_*( [\Omega_w]\cdot [\tilde\Omega_{\bar{w_\circ v}}] ) = \delta_{w,v}.
\]
(In fact, $[\tilde\Omega_{w}] = [\Omega_w]$ in $H^*\Fl(\nn)$, but since this does not hold in equivariant cohomology, we prefer to use distinct notation.)  To emphasize this duality, we often write $w^\vee = \bar{w_\circ w}$.  Note that $\ell(w^\vee) = \dim \Fl(\nn)-\ell(w)$.

\subsection{Equivariant cohomology}\label{ss:eqcoh}

Let $T\isom (\CC^*)^n$ be a torus.  One can find a contractible space $\EE{T}$ on which $T$ acts freely, and the quotient $\BB{T}=\EE{T}/T$ is then unique up to homotopy.  The \emph{equivariant cohomology} of a space $X$ equipped with a $T$-action is defined by
\[
  H_T^*X = H^*(\EE{T} \times^T X),
\]
where $Y\times^T Z$ denotes the quotient of $Y\times Z$ by the relation $(y\cdot t, z)\sim (y,t\cdot z)$.  The map $\EE{T}\times^T X \to \BB{T}$ makes $H_T^*X$ an algebra over
\[
  \Lambda_T = H_T^*(\pt) = H^*(\BB{T}) \isom \ZZ[t_1,\ldots,t_n].
\]

The spaces $\EE{T}$ are infinite-dimensional, but one can find finite-dimensional algebraic varieties which serve as ``approximations.''  To describe these, we will need to pay attention to the isomorphism $T\isom(\CC^*)^n$.  Let $M = \Hom(T,\CC^*)$ be the character group of $T$, so $M\isom\ZZ^n$, and there is the \emph{standard basis} $t_1,t_2,\ldots,t_n$.  We will also use the \emph{positive basis}
\[
  \alpha_1, \ldots, \alpha_{n-1}, \alpha_n
\]
for $M$, where $\alpha_i = t_i-t_{i+1}$ for $i<n$, and $\alpha_n=t_n$.  (The reason for this choice will become evident in \S\ref{s:mixing}.)  

Now take $\EE$ to be $(\CC^N\setminus\{0\})^n$, for $m\gg 0$.  With $T$ acting on the $n$ factors via the positive basis, we set
\[
  \BB = \EE/T = (\PP^{N-1})^n.
\]
These spaces approximate $\EE{T}\to\BB{T}$ in the sense that
\[
  H^k(\EE\times^T X) = H^k(\EE{T}\times^T X) = H_T^kX
\]
for all sufficiently small $k$, and there are compatible inclusion maps as $N\to\infty$.  See \cite{eg} for details on approximation spaces in equivariant cohomology; the key point for our purposes is that one can carry out any given computation in $H_T^*X$ using $H^*(\EE\times^T X)$.

As a matter of notation, given a $T$-space $X$, we will denote the corresponding approximation space by a bold letter $\Xx=\EE\times^T X$, always understanding some fixed $N\gg 0$.

We will also need to consider certain linear subspaces of $\BB$.  Specifically, for each integer $j$ with $0\leq j\leq N-1$, fix transverse linear subspaces $\PP^{N-1-j}$ and $\tilde\PP^{j}$ inside $\PP^{N-1}$, and for a multi-index of such integers $J = (j_1,\ldots,j_n)$, set
\[
  \BB^J = \PP^{N-1-j_1} \times \cdots \times \PP^{N-1-j_n} \quad \text{ and } \quad 
  \BB_J = \tilde\PP^{j_1} \times \cdots \times \tilde\PP^{j_n}.
\]
Thus $\dim\BB_J = \codim(\BB^J,\BB) = |J| = j_1+\cdots+j_n$.  These subspaces carry the effective equivariant classes in $H_T^*(\pt)$; hence their significance:
\begin{align}\label{e:t-class}
 [\BB^J] &= (-\alpha_1)^{j_1} \cdots  (-\alpha_n)^{j_n} \quad \text{ in }\quad H^*\BB = H_T^*(\pt) .
\end{align}

Finally, let $\Xx_J$ and $\Xx^J$ denote the preimages of $\BB_J$ and $\BB^J$, respectively, under the projection $\Xx \to \BB$.  

Let $\pi^\BB$ be the map $\BB\to\pt$.  Note that any polynomial $c(t) \in H_T^*(\pt) = \ZZ[t_1,\ldots,t_n]$ can be written as
\begin{equation}\label{e:coefficient}
  c(t) = \sum_J c_J \, (-\alpha_1)^{j_1}\cdots(-\alpha_{n})^{j_{n}},
\end{equation}
where $c_J = \pi^\BB_*( c(t)\cdot[\BB_J] )$.  (This is just Poincar\'e duality on $\BB$.)

In particular, we have:
\begin{lemma}\label{l:coefficient}
Suppose $c(t) = \pi^T_*(\sigma)$, for some class $\sigma \in H_T^*X$, where $\pi^T_* \colon H_T^*X \to H_T^*(\pt)$ is the equivariant pushforward.  Identify $\pi^T$ with the corresponding projection $\Xx \to \BB$, and let $\pi^{\Xx}$ be the map $\Xx \to \pt$.  Then the coefficient $c_J$ appearing in \eqref{e:coefficient} is equal to $\pi^{\Xx}_*(\sigma \cdot [\Xx_J])$.
\end{lemma}

In the present context, $T\isom (\CC^*)^n \subset GL_n$ is the maximal torus of diagonal matrices, acting on $\Fl(n)$.  The equivariant cohomology of the complete flag variety has a well-known presentation
\[
  H_T^*\Fl(\CC^n) = \Lambda_T[x_1,\ldots,x_n]/(e_1(x)-e_1(t),\ldots,e_n(x)-e_n(t)),
\]
where $x_i = c^T_1(\ker(Q_i \to Q_{i-1}))$, and $e_i$ is the $i$th elementary symmetric function.  
More generally, as in \eqref{e:sigma-def1}, define
\begin{equation}\label{e:sigma-def}
  \sigma_i^j=c^T_i(\ker(Q_j\rightarrow Q_{j-1})).
\end{equation}
For the partial flag variety $\flag$, we have
\[
  H_T^*\Fl(\nn) = \Lambda_T[\sigma_1^1,\ldots,\sigma_{n-n_m}^{m+1}]/(\tilde{e}_1(m+1))-e_1(t),\ldots,\tilde{e}_n(m+1)-e_n(t)).
\]

Moreover, the equivariant classes of Schubert varieties form a $\Lambda$-basis for $H_T^*\flag$, and for essentially the same reason as in the classical case, the classes of opposite Schubert varieties are the Poincar\'e dual basis:
\begin{equation}\label{e:eq-duality}
  \pi^T_*( [\Omega_w]^T \cdot [\tilde\Omega_{v^\vee}]^T ) = \delta_{w,v}
\end{equation}
in $\Lambda$.  (This is a stronger statement, since a priori, these classes could pair to a nonzero polynomial in $t$.)

The equivariant Giambelli formula is given by \emph{double Schubert polynomials} $\Sch_w(x,t)$, defined by 
\[
  \Sch_w(x,t) = \sum_{u,v} (-1)^{\ell(v)} \Sch_u(x)\Sch_v(t),
\]
where the sum is over $u,v\in S_{n+1}$ such that $v^{-1}u=w$ and $\ell(u)+\ell(v)=\ell(w)$.  A key property of the double polynomials is that $\partial^t_i\Sch_w(x,t) = -\Sch_{s_i w}(x,t)$ whenever $\ell(s_i w)<\ell(w)$, where $\partial^t_i$ is the divided difference operator (defined in \eqref{e:divdiff}) applied to the $t$ variables.

The equivariant Giambelli formula gives $[\Omega_w]^T$ in terms of the presentation (see, e.g., \cite{knmi}).  For the complete flag variety $\Fl(\CC^n)$, we have 
\[
  [\Omega_w]^T=\Sch_w(x,t) \text{ in } H_T^*(\Fl(\CC^n)).
\]
For $w\in\Sn$, $\Sch_w(x,t)$ can be written as a polynomial in $\sigma_i^j$ and $t_i$, which we write as $\Sch_w^\nn(\sigma,t)$, and for partial flag varieties, we have

\begin{theorem}
For any $w\in\Sn$, $[\Omega_w]^T=\Sch_w^\nn(\sigma,t)$ as classes in $H_T^*(\Fl(\nn))$.
\end{theorem}

One can realize $\FFl(\nn) = \EE\times^T\Fl(\nn)$ as a flag bundle $\FFl(\nn;E) \to \BB$, for a vector bundle $E$ on $\BB$.  Specifically, let $pr_i:\BB\to\PP^{N-1}$ be the projection on the $i$th factor, and let $L_i = pr_i^*\OO(-1)$.  (One has $t_i=c_1(L_i)$ in $H_T^2(\pt) = H^2(\BB)$.)  Then set $E_i=L_1\oplus L_2 \oplus \cdots \oplus L_i$ and $E=E_n$, so we obtain a flag of vector bundles $E_\bullet$.  The equivariant Schubert class $[\Omega_w]^T$ is identified with the class of the degeneracy locus $\OOmega_w \subseteq \FFl(E)$.

\subsection{Quantum cohomology}

The \emph{(small) quantum cohomology ring} $QH^*\Fl(\nn)$ is a commutative and associative graded algebra over $\ZZ[\ul{q}]=\ZZ[q_1,\ldots,q_m]$, where $q_i$ is a parameter of degree $n_{i+1}-n_{i-1}$.  As a module, $QH^*\Fl(\nn)$ is simply $\ZZ[\ul{q}]\otimes_\ZZ H^*\Fl(\nn)$, so it has a $\ZZ[\ul{q}]$-basis of Schubert classes $\scl_w$:
\[
  QH^*\Fl(\nn) = \bigoplus_{w\in \Sn} \ZZ[\ul{q}]\cdot \scl_w.
\]
The quantum product is a deformation of the usual cup product.  For permutations $u,v$, define a product by
\begin{equation}
  \scl_u \qp \scl_v = \sum_{w,\dd} \ul{q}^{\dd}\, c_{u,v}^{w,\dd}\, \scl_w,
\end{equation}
where $\dd$ ranges over $(n-1)$-tuples of nonnegative integers.  The \emph{quantum Littlewood-Richardson coefficient} $c_{u,v}^{w,\dd}$ is a \emph{three-point Gromov-Witten invariant}; it may be interpreted informally as the number of maps $f:\PP^1\to\Fl(\nn)$ of degree $\ul{d}$ such that $f(0)$, $f(1)$, and $f(\infty)$ lie in general translates of $\Omega_u$, $\Omega_v$, and $\Omega_{w^\vee}$, respectively.  (A map has \emph{degree} $\ul{d}$ if $f_*[\PP^1] =  d_1\scl_{s_{n_1}^\vee} + \cdots + d_{m}\scl_{s_{n_m}^\vee}$.  Since the classes $\scl_{s_{n_i}^\vee}$ form a basis for $H_2\Fl(\nn)$, this is well-defined.)

The precise definition of $c_{u,v}^{w,\ul{d}}$ is usually phrased in terms of the Kontsevich moduli space of stable maps.  In order to set up notation, we sketch the construction here.  (See, e.g., \cite{fp} for details.)  There is a smooth, proper Deligne-Mumford stack
\[
  \Mbar_{r}(\dd)=\Mbar_{0,r}(\Fl(\nn),\dd),
\]
called the \emph{(Kontsevich) space of stable maps}, which parametrizes data \linebreak $(f,C,\{p_1,\ldots,p_r\})$, where $C$ is a genus-zero curve with marked points $p_1,\ldots,p_r$, and $f:C\to\Fl(\nn)$ is a map of degree $\ul{d}$, and a certain stability condition is imposed.  The space of stable maps has dimension equal to $\dim \Fl(\nn) + \sum_{i=1}^m d_i(n_{i+1}-n_{i-1}) + r-3$, and its coarse moduli space (with which we will tacitly work) is a Cohen-Macaulay projective variety.

This space of stable maps comes with natural evaluation morphisms
\[
  \ev_i:\Mbar_{r}(\dd) \to \Fl(\nn),
\]
for $i=1,\ldots,r$, defined by sending $(f,C,\{p_j\})$ to $f(p_i)$.  It also has a forgetful morphism
\[
 f:\Mbar_{r}(\dd) \to \Mbar_{0,r}
\]
to the space of stable curves, which is a smooth projective variety of dimension $r-3$.

The quantum product is defined using $\Mbar_{3}(\dd)$.  Write $\pi:\Mbar_3(\dd)\to\pt$ for the map to a point.  Now one defines
\begin{equation}\label{e:mbar-defn}
 c_{u,v}^{w,\dd} = \pi_*( (\ev_1^*\scl_u)\cdot(\ev_2^*\scl_v)\cdot(\ev_3^*\scl_{w^\vee}) ),
\end{equation}
where $\alpha\cdot\beta$ denotes the usual cup product in $H^*\Mbar_{3}(\dd)$.

\subsection{Equivariant quantum cohomology}

When a torus $T$ acts on $\Fl(\nn)$, there is an induced action on $\Mbar_{r}(\dd) = \Mbar_{0,r}(\Fl(\nn),\dd)$, so one can define a quantum deformation of $H_T^*\Fl(\nn)$ analogously to the classical case.  Let $\pi^T_*:H_T^*\Mbar_{3}(\dd) \to H_T^*(\pt)$ be the equivariant pushforward.  One defines a product on $QH_T^*\Fl(\nn) = \Lambda[\qq]\otimes_\Lambda H_T^*\Fl(\nn)$ by
\[
  \scl_u \qtp \scl_v = \sum_{w,\dd} \qq^\dd\, c_{u,v}^{w,\ul{d}}(t)\, \scl_w,
\]
where the coefficient is a \emph{three-point equivariant Gromov-Witten invariant}
\[
  c_{u,v}^{w,\dd}(t) = \pi^T_*( \ev_1^*\scl_u \cdot \ev_2^*\scl_v \cdot \ev_3^*\tilde\scl_{w^\vee} ).
\]
(In contrast to the non-equivariant case, it is important to use the \emph{opposite} Schubert class $\tilde\scl_{w^\vee}$ as the third insertion, rather than $\scl_{w^\vee}$.)

As before, this defines an associative product \cite{kim-eq}.  Following \cite[\S5]{mihalcea-positivity}, we will call the polynomials $c_{u,v}^{w,\dd}$ \emph{equivariant quantum Littlewood-Richardson (EQLR) coefficients}, and use this term also for the coefficients defined by associativity:
\begin{equation*}
  \scl_{v_1} \qtp \scl_{v_2} \qtp \cdots \qtp \scl_{v_r} = \sum_{w,\ul{d}} \qq^{\dd} \, c_{v_1,\ldots,v_r}^{w,\dd}(t) \, \scl_w .
\end{equation*}
The proof of associativity given in \cite[\S3.3]{kim-eq} shows that the EQLR coefficients may be described equivalently as
\begin{equation}\label{e:eqlr}
 c_{v_1,\ldots,v_r}^{w,\dd}(t) = \pi^T_*( \ev_1^*\scl_{v_1}\cdots \ev_r^*\scl_{v_r}\cdot \ev_{r+1}^*\tilde\scl_{w^\vee} \cdot f^*[\pt]),
\end{equation}
where $f:\Mbar_{r+1}(\dd) \to \Mbar_{0,r+1}$ is the forgetful map and $[\pt] \in H_T^*\Mbar_{0,r+1}$ is the class of a point, with $T$ acting trivially on the space of stable curves.

\begin{remark} 
The \emph{(r+1)-point equivariant Gromov-Witten invariant} is defined via $r+1$ evaluation maps from $\Mbar_{0,r+1}(\Fl(\nn),\dd)$ to $\Fl(\nn)$ as
\[
 \pi^T_*( \ev_1^*\scl_{v_1}\cdots \ev_r^*\scl_{v_r}\cdot \ev_{r+1}^*\tilde\scl_{w^\vee} ).
\]
While the EQLR coefficients for $r=2$ agree with the corresponding three-point invariants, in general $c_{v_1,\ldots,v_r}^{w,\dd}\neq \pi^T_*( \ev_1^*\scl_{v_1}\cdots \ev_r^*\scl_{v_r}\cdot \ev_{r+1}^*\tilde\scl_{w^\vee} )$.  The situation is the same in the non-equivariant case; see, e.g., \cite[\S10]{fp}.
\end{remark}

A presentation of the ring $QH_T^*\flag$, specializing to one for $QH^*\flag$, is given in \cite[Theorem~2]{kim-eq}.  We will give a different proof of this in Corollary~\ref{c:presentation}.

\section{Universal Schubert polynomials}\label{s:univ-sch}

%
\subsection{Definitions}

Universal double Schubert polynomials were introduced in \cite{fulton} as the solution to a certain degeneracy locus problem.  They specialize to double Schubert polynomials as well as to quantum Schubert polynomials for complete and partial flag varieties \cite{fgp}\cite{cf-partial}.  To describe the universal double Schubert polynomials $\Sch_w(c,d)$, for $w\in S_n$, we first give two formulations of \emph{universal (single) Schubert polynomials}: $\Sch_w(c)$ and $\Sch_w(g)$.  The first form, denoted $\Sch_w(c)$, is a polynomial in independent variables $c_k(l)$ of degree $k$, for $1\leq k\leq l\leq n$; we set $c_0(l)=1$ and $c_k(l)=0$ when $k<0$ or $k>l$.  The second form, denoted $\Sch_w(g)$,
is a polynomial in variables $g_i[j]$ for $i,j\geq 0$, and $i+j\leq n$, with $g_i[j]$ of degree $j+1$.

For $w\in S_{n+1}$, the classical Schubert polynomial $\Sch_w(x)$ can be written uniquely as
\[
  \Sch_w(x) = \sum a_{k_1\ldots k_n}e_{k_1}(1) \cdot\cdots\cdot e_{k_n}(n),
\]
where the sum ranges over sequences $(k_1,\ldots,k_n)$ with $0\leq k_p\leq p$ and $\sum k_p=\ell(w)$,
and where $e_k(l):=e_k(x_1,\ldots,x_l)$ is the $k$th elementary symmetric polynomial in the variables $x_1,\ldots,x_l$.  
Define the universal Schubert polynomial by
\begin{equation}\label{sch-c}
 \Sch_w(c) = \sum a_{k_1\ldots k_n}c_{k_1}(1)\cdot \cdots \cdot c_{k_n}(n).
\end{equation}
When $c_k(l)$ is specialized to $e_k(l)$, the polynomial $\Sch_w(c)$ becomes the classical Schubert polynomial $\Sch_w(x)$. 

The second formulation of universal Schubert polynomials $\Sch_w(g)$ is as follows.  Label the vertices of the Dynkin diagram $(A_{n})$ by $x_1,\ldots, x_n$, and label the edges $g_1[1],\ldots,g_{n-1}[1]$, where $g_i[1]$ connects $x_i$ and $x_{i+1}$.  Now denote by $g_i[j]$ the path covering the $j+1$ consecutive vertices $x_i,\ldots,x_{i+j}$, and define $E_k^l(g)$ to be the sum of all monomials in paths $g_i[j]$ covering exactly $k$ of the verticies $x_1,\ldots, x_l$ with no vertex covered more than once.  When the variables $g$ are understood, we may simply write $E_k^l$.  

Alternatively, consider the $l\times l$ matrix $M_l$ with $g_a[b-a]$ in the $(a,b)$th entry for $1\leq a\leq b\leq l$, $-1$ in the $(a+1,a)$ entries below the diagonal, and zero elsewhere.  Define a polynomial $E_k^l(g)$ in the variables $g_i[j]$ as the coefficient of $T^k$ in the determinant of $M_l+IT$, setting $g_i[0]=x_i$.   Both this and the description of $E_k^l$ in the previous paragraph are equivalent to the inductive definition:
\begin{equation}\label{e:Ekl-inductive}
  E_k^l(g)=E_k^{l-1}(g) + \sum_{j=0}^{k-1} E_{k-j-1}^{l-j-1}(g)\,g_{l-j}[j].
\end{equation}
The universal Schubert polynomial $\Sch_w(g)$ is obtained by substituting $c_k(l)=E_k^l(g)$ into the expression \eqref{sch-c} for $\Sch_w(c)$.

We can now define \emph{universal double Schubert polynomials} in variables $c_k(l)$ and $d_k(l)$ by
\begin{equation}\label{e:univ-double-def}
  \Sch_w(c,d) = \sum_{u,v} (-1)^{\ell(v)} \Sch_u(c)\,\Sch_v(d),
\end{equation}
where the sum is over $u,v\in S_{n+1}$ such that $v^{-1}u=w$ and $\ell(u)+\ell(v)=\ell(w)$.  These polynomials can also be written as $\Sch_w(g,h)$, using variables $g_i[j]$ and $h_i[j]$ obtained by substituting $c_k(l)=E_k^l(g)$ and $d_k(l)=E_k^l(h)$ into $\Sch_w(c,d)$.  Upon setting $h_i[ 0]=y_i$ and $h_i[j]=0$ for all $j>0$, $\Sch_w(c,d)$ specializes to polynomials $\Sch_w(c,y)$; this is equivalent to specializing  $d_k(l)$  to the elementary symmetric polynomial $e_k(y_1,\ldots,y_l)$.  

For the purposes of this paper, we will focus on the specialized double Schubert polynomial $\Sch_w(c,y)$ and its alternative form $\Sch_w(g,y)$, obtained by the substitution $c_k(l)=E_k^l(g)$ into $\Sch_w(c,y)$.  These specializations can be computed inductively from a ``top'' polynomial, using divided difference operators, as in the classical case.  Specifically, for the longest permutation $w_\circ$ in $S_{n+1}$, we have
\[
  \Sch_{w_\circ}(c,y) = \prod_{i=1}^n \left(\sum_{j=0}^i c_{i-j}(i) (-y_{n+1-i})^j\right),
\]
and $\Sch_{s_i w}(c,y) = -\partial^y_i\Sch_w(c,y)$ whenever $\ell(s_iw)<\ell(w)$, where $\partial^y_i$ is the divided difference operator applied to the $y$ variables \cite[{eqs.~(7) and (8)}]{fulton}; see also \cite[pp.~502--503]{cf-partial} and \cite{kima}.  (In a similar fashion, one can also compute the unspecialized versions inductively, but the analogues of divided difference operators are a little more complicated.)

When $c_k(l)$ is also specialized to $e_i(x_1,\ldots,x_l)$, $\Sch_w(c,d)$ becomes the classical double Schubert polynomial $\Sch_w(x,y)$.

\begin{example} For $w=312$, we have
\begin{align*}
\Sch_{312}(x)&= x_1^2= x_1(x_1+x_2)-x_1x_2 = e_1(1)e_1(2)-e_2(2)\\
\Sch_{312}(c)&= c_1(1)c_1(2)-c_2(2)\\
\Sch_{312}(g)&=  x_1(x_1+x_2)-(x_1x_2+g_1[1])\\ 
\Sch_{312}(c,d) &= c_1(1)c_1(2)-c_2(2)-c_1(2)d_1(2)-d_2(2) 
\end{align*}
\end{example}

\begin{remark}
\label{r:c-g-equiv}
The inductive relation \eqref{e:Ekl-inductive} can be inverted to express the variables $g_i[j]$ as polynomials in $c_k(l)$.  In other words, the variables $c_k(l)$ and $g_i[j]$ generate the same polynomial rings, $\ZZ[g]=\ZZ[c]$.  Therefore $\ZZ[g,y]=\ZZ[c,y]$.  For $1\leq k\leq n_l$, the cyclic permutation $\alpha_{k,l} := s_{n_l-k+1}\cdots s_{n_l}$, has universal (single) Schubert polynomial $\Sch_{\alpha_{k,l}}(c) = c_k(l)$, and therefore
$\ZZ[g,y] = \ZZ[\Sch_{\alpha_{k,l}}(c),y]$.  Moreover, the corresponding universal double Schubert polynomials can be written  
\begin{align*}
\Sch_{\alpha_{k,l}}(c,y) &= \sum_{0\leq k'\leq k} (-1)^{k-k'}\Sch_{\alpha_{k',l}}(c)\, \Sch_{s_{n_l-k+1}\cdots s_{n_l-k'}}(y)\\
&=  c_k(l) + \sum_{0\leq k'<  k} (-1)^{k-k'} c_{k'}(l)\, h_{k-k'}(y),
\end{align*}
where $h_k(y)$ is the $k$th complete symmetric polynomial in $y_1,\ldots, y_n$.  Therefore we can recursively write each $c_k(l)$ in terms of polynomials $\Sch_{\alpha_{k,l}}(c,y)$ and variables $y$, and therefore $\ZZ[g,y]=\ZZ[c,y]=\ZZ[\Sch_{\alpha_{k,l}}(c,y),y]$.

In particular, each $g_i[j]$ can be written as a polynomial in the $\Sch_{\alpha_{k,l}}(c,y)$ and $y_1,\ldots,y_n$, with $k\leq j+1$.  
\end{remark}

\begin{remark}\label{r:partial-double}
Several properties of certain of the polynomials $\Sch_w(c,y)$ will be useful in studying partial flag varieties.  Let $\nn$ and $\Sn$ be as in Section \ref{s:back}.
We define polynomials $\Sch_w^{\nn}(g,y)$ as follows. For a permutation $w\in \Sn$, let $\Sch_w^{\nn}$ be the result after setting $g_i[0]=x_i$ and $g_i[j]=0$ if $j>0$ and $i+j\neq n_p$ for some $p$. An alternative definition of $\Sch_w^{\nn}$ is given by performing the substitutions $c_k(l)\to c_k(n_p)$ for $l\in[n_p,n_{p+1})$ into $\Sch_w(c,y)$ and then performing the above substitutions for the $g_i[j]$. The proof of \cite[Proposition 4.3]{fulton} shows that these two constructions yield the same 
$\Sch^\nn_w(g,y)$.  
\end{remark}

Remarks \ref{r:c-g-equiv} and \ref{r:partial-double} yield the following useful lemma.

\begin{lemma}
\label{l:rewrite-g}
Each  $g_i[j]$  can be written as a polynomial in $\Sch_w(g,y)$ and $y_1,\ldots,y_n$, with  $\ell(w)\leq j+1$.  Moreover, when $i+j\in\nn$, $g_i[j]$ can be written as a polynomial in $\Sch^\nn_w(g,y)$ and $y_1,\ldots,y_n$, where $w$ is a permutation in $\Sn$ with $\ell(w)\leq j+1$.
\end{lemma}

\subsection{Degeneracy locus formula}

Fulton proves that universal double Schubert polynomials give the answer to a degeneracy locus problem.  While not explictly stated in \cite[Theorem 3.7]{fulton}, the formula holds equivariantly. 

Let $T$ act on an algebraic Cohen-Macaulay scheme $X$, and consider maps of equivariant vector bundles
\[
 E_1\rightarrow \cdots \rightarrow E_n \rightarrow F_n \rightarrow \cdots \rightarrow F_1,
\] 
where $E_i$ and $F_i$ are of rank $i$.  Let $\Omega_w$ be the degeneracy locus
\[
 \Omega_w = \{x\in X | \rank_x(E_q\rightarrow F_p)\leq r_w(p,q) \text{ for all }1\leq p,q\leq n\},
\]
and let $\Sch_w(c^T(F_\bullet),c^T(E_\bullet))$ be the image of $\Sch_w(c,d)$ by specializing $c_k(l)$ to the equivariant Chern class $c^T_k(F_l)$, and specializing $d_k(l)$ to $c^T_k(E_l)$.

\begin{theorem}\label{t:fulton}
For $w\in S_n$, we have $[\Omega_w] = \Sch_w(c^T(F_\bullet),c^T(E_\bullet))$ in $H_T^*(X)$ whenever $\codim_X(\Omega_w) = \ell(w)$.
\end{theorem}

\begin{proof}
Apply \cite[Theorem~3.7]{fulton} to the corresponding degeneracy locus on the mixing space $\Xx$.
\end{proof}

We will be interested in the following situation.  For a sequence of integers $\nn = \{0<n_1<\cdots <n_m<n\}$, consider maps of vector bundles
\begin{equation}\label{e:maps-partial}
  E_1\rightarrow \cdots \rightarrow E_n \rightarrow F_m \rightarrow \cdots \rightarrow F_1
\end{equation}
where $\rank(F_p)=n_p$, and $E_i \isom \CC^i \otimes \OO_X$ is trivial, but has the nontrivial equivariant structure coming from the diagonal action of $T$ on $\CC^n$.  Let $\Omega_w$ be the degeneracy locus
\[
  \Omega_w = \{x\in X | \rank_x(E_q\rightarrow F_p)\leq r_w(n_p,q) \text{ for all } p,q \}.
\]
For $w\in \Sn$, let $\Sch_w^{\nn}(c^T(F_\bullet),t)$ denote the result of specializing $c_k(n_p)$ to $c^T_k(F_p)$ in $\Sch^\nn_w(c,t)$.  We obtain
 
\begin{corollary}\label{c:fulton-partial}
Given maps as in (\ref{e:maps-partial}), for $w\in \Sn$, we have $[\Omega_w]^T = \Sch_w^{\nn}(c^T(F_\bullet),t)$ in $H_T^*(X)$.
\end{corollary}



Fix $\nn$.  As in Remark~\ref{r:c-g-equiv}, the polynomial ring generated by $c_k(n_p)$ (for $1\leq k\leq n_p$ and $1\leq p\leq m$) is equal to the polynomial ring generated by $g_i[j]$ (for $i+j=n_p$ and $1\leq p\leq m$).  Using the corresponding identification $\ZZ[c,y]=\ZZ[g,y]$, consider the map $b: \ZZ[c,y] \to H_T^*X$ defined by $c_k(n_p) \mapsto c^T_k(F_p)$ and $y_i \mapsto t_i$.  We define classes $Q_i[j]$ in $H_T^{2(j+1)}(X)$ by
\begin{equation}\label{e:equiv-CQ}
  Q_i[j] = b(g_i[j]).
\end{equation}
With this notation, we can write $[\Omega_w]^T= \Sch_w^{\nn}(Q,t)$ in $H_T^*(X)$.  Since $c^T_k(F_l)$ is symmetric in $x_{n_{j-1}+1},\ldots,x_{n_j}$ for every $1\leq j\leq l$, it can be written as a polynomial in variables $\sigma_i^j$ and $Q_i[j]$, for $i+j=n_p$ and $1\leq p \leq m$.  We denote this polynomial by $\tilde{E}_k^l(\sigma)$ or $\widetilde{E}_k^l(Q)$.


With this $\tilde{E}_k^l(Q)$, and $Q_i[j]$ as defined in \eqref{e:equiv-CQ}, we have the following:

\begin{lemma}
\label{l:chern-quotient}
Given equivariant maps of vector bundles $F_{m+1} \rightarrow \cdots \rightarrow F_1$
with $\rank(F_i)=n_i$ and $c_k(F_l)=\tilde{E}_k^l(Q)$, for $1\leq i\leq n_{l+1}-n_l$, we have
\[
c^T_i(\ker(F_{l+1}\rightarrow F_l)) = Q_{n_{l+1}-i+1}[i-1]
\]
for $n_{l+1}-n_l<i\leq n_{l+1}-n_{l-1}$.
\end{lemma}

Moreover, if $F_b\rightarrow F_a$ is a surjection of vector bundles for some $a<b$, then $F_{b'}\rightarrow F_a$ is also a surjection of bundles for $a\leq b'\leq b$.  From Lemma \ref{l:chern-quotient}, \cite[Proposition 6.2]{chen}, and Remark \ref{r:partial-double}, we obtain

\begin{lemma}\label{l:Q-zero} If $F_b\rightarrow F_a$ is a surjection of vector bundles, then $Q_i[j]=0$ for all $i<n_a+1\leq i+j\leq n_b$.

\end{lemma}

For $1\leq a\leq b$, define $H_a^b(g)$ to be the polynomial obtained by the substitution $c_k(l)=E_k^l(g)$ into $\det \left( c_{1+j-i}(b+j-1)\right)_{1\leq i,j\leq a}$. This is the universal analogue of the complete symmetric polynomial $h_a(x_1,\ldots,x_b)$.

\begin{remark}
\label{r:special-usp} For $1\leq k\leq n-n_l$, write $\beta_{k,l}$ for the cyclic permutation $\beta_{k,l}:=s_{n_l+k-1}\cdots s_{n_l}$.  This is a Grassmannian permutation with descent at $n_l$ (of length $k$), so by Proposition 4.4 of \cite{fulton}, its (single) universal Schubert polynomial is $\Sch_{\beta_{k,l}}(c) = \det \left( c_{1+j-i}(n_l+j-1)\right)_{1\leq i,j\leq k}$, so that $\Sch_{\beta_{k,l}}(g)=H_k^{n_l}(g)$.  Note that $\alpha_{k,l}$ and $\beta_{k,l}$ are permutations in $\Sn$, and by Remark~\ref{r:partial-double}, for $0\leq k\leq n_{l+1}-n_l$, we have $\Sch^\nn_{\beta_{k,l}}(c)= \det \left( c_{1+j-i}(n_l)\right)_{1\leq i,j\leq k}$.

Moreover, given maps of vector bundles $F_{m+1} \rightarrow \cdots \rightarrow F_1$
with $\rank(F_i)=n_i$ as above, for $0\leq k\leq n_{l+1}-n_l$, by expanding $\det \left( c_{1+j-i}(n_l)\right)_{1\leq i,j\leq k}$ along the top row, we obtain inductively that
 $\Sch^\nn_{\beta_{k,l}}(c)= (-1)^k c_k(-F_l)$.

\end{remark}

\section{Quot schemes and spaces of maps}

Recall that the EQLR coefficients are defined using the Kontsevich compactification of the space of maps $f:\PP^1\to\Fl(\nn)$ of degree $\dd$.  
Using the fact that a degree-$\dd$ map $\PP^1 \to \Fl(\nn)$ corresponds to successive quotient bundles of $V^*_{\PP^1}=V^*\otimes \OO_{\PP^1}$ of rank $n_i$ and degree $d_i$, there is another compactification of $\mor_{r+1}(\dd)$.  The \emph{hyperquot scheme} $\hq_{\dd}$ parametrizes flat families of successive quotient sheaves of $V^*_{\PP^1}$ of rank $n-n_i$ and degree $d_i$, generalizing Grothendieck's Quot scheme.  (We consider quotients of $V^*$ rather than of $V$ for technical reasons.)  It is a smooth, projective variety of dimension $\dim\Fl(\nn)+\sum d_i(n_{i+1}-n_{i-1})$, and comes with a universal sequence of quotient sheaves on $\PP^1\times\hq_{\dd}$
\[
 V^*_{\PP^1\times\hq_{\dd}} = \Bb_{m+1} \twoheadrightarrow \cdots \twoheadrightarrow \Bb_2 \twoheadrightarrow \Bb_1,
\]
where $\Bb_i$ has rank $n-n_i$ and relative degree $d_i$.  
The space of maps $\mor(\dd)=\mor_3(\dd)$ embeds in $\hq_\dd$ as the largest (open) subscheme such that each restriction $\Bb_i|_{\PP^1\times\mor(\dd)}$ is locally free.

To avoid torsion, we prefer to work with the locally free sheaves $\Aa_i:= \ker(V^*_{\PP^1\times\hq_{\dd}} \rightarrow \Bb_i)$ which are locally free of rank $n_i$ and degree $-d_i$. There is a sequence
\[
  \Aa_1 \to \Aa_2 \to \cdots \to \Aa_{m+1} = V^*_{\PP^1\times\hq_{\dd}}
\]
The tradeoff here is that although the map  $\Aa_i \to \Aa_{i+1}$ is an inclusion of sheaves, it is not necessarily an inclusion of vector bundles (i.e., the cokernel is not locally free in general).


Dualizing again, and regarding the standard flag $E_\bullet$ as a flag of trivial vector bundles on $\PP^1\times\hq_{\dd}$, we have a sequence
\begin{equation}\label{e:sequence}
  E_1\hookrightarrow \cdots \hookrightarrow E_{n-1}\hookrightarrow E_n = V_{\PP^1\times\hq_{\dd}} = \Aa_{m+1}^* \rightarrow \Aa_{m}^*\rightarrow \cdots \rightarrow \Aa_1^*.
\end{equation}
Note that the maps $\Aa_{i+1}^* \to \Aa_i^*$ are not necessarily surjective.

Define $\D_w \subseteq \PP^1\times \hq_\dd$ as the degeneracy locus associated to this sequence:
\begin{equation}\label{e:dw}
  \D_w = \{  x \,|\, \rk_x(E_q \to \Aa_p^*) \leq r_w(p,q) \text{ for all } 1\leq q\leq n, n_p\in\nn\}.
\end{equation}
Fix a point $z\in\PP^1$, and let $\D_w(z) = \D_w \cap (\{z\}\times\hq_\dd)$ be the corresponding closed subscheme of $\hq_\dd$.  Define $\tilde{\D}_w$ and $\tilde{\D}_w(z)$ similarly, using the opposite flag $\tilde{E}_\bullet$ in place of $E_\bullet$.

For a point $z\in\PP^1$, there is an evaluation map $\ev_z:\mor(\dd) \to \Fl(\nn)$, defined by the sequence of quotient sheaves $\Aa_i^*$ restricted to $\mor(\dd)=\{z\}\times \mor(\dd)$.  We write $\Omega_w(z) = \ev_z^{-1}(\Omega_w)$ and $\tilde\Omega_w(z) = \ev_z^{-1}(\tilde\Omega_w)$.  From the definitions, one sees
\begin{equation}\label{e:omegad}
\Omega_w(z) = D_w(z)\cap \mor(\dd) \quad \text{ and } \quad \tilde\Omega_w(z) = \tilde{D}_w(z)\cap \mor(\dd).
\end{equation}

The group $GL_n=GL(V)$ acts on $\hq_\dd$ via its action on $V$: a quotient sheaf of $V_{\PP^1}$ is sent to the quotient obtained by precomposing with an automorphism of $V$.  The above constructions are all equivariant for appropriate subgroups; in particular, we set
\[
  \mu_w=[\D_{w}(z)]^T \quad \text{ and } \quad \tilde{\mu}_w=[\tilde{\D}_w(z)]^T
\]
in $H_T^*(\hq_\dd)$.  Since $T$ acts trivially on the $\PP^1$ factor of $\PP^1\times\hq_\dd$, these equivariant classes are independent of the choice of $z$.

Write $A_l$ for the restriction of $\Aa_l$ to $\{z\}\times\hq_\dd$.  (The choice of point $z\in\PP^1$ will usually be irrelevant.)  The degeneracy locus formula of Corollary \ref{c:fulton-partial} yields the following:

\begin{proposition}\label{p:deg-mu}
Setting $c_k(l) = c^T_k(A_l^*)$, we have $\mu_w = \Sch^\nn_w(c,t)$ in $H_T^*\hq_\dd$.
\end{proposition}

We also have interpretations of the classes $Q_i[j]$.  In this context, Lemma~\ref{l:chern-quotient} says
\begin{equation*}
  Q_{n_{l+1}-i+1}[i-1] = c^T_i(\ker(\Aa^*_{l+1}\rightarrow \Aa^*_l)),
\end{equation*}
for $n_{l+1}-n_l<i\leq n_{l+1}-n_{l-1}$, so in particular,
\begin{equation}\label{e:q-vars}
  Q_{n_{l-1}+1}[n_{l+1}-n_{l-1}-1] = c^T_{n_{l+1}-n_{l-1}}(\ker(\Aa^*_{l+1}\rightarrow \Aa^*_l)).
\end{equation}
In the case of the complete flag variety, $Q_l[1]$ is the equivariant class of the locus where $\Aa^*_{l+1} \to \Aa^*_{l}$ fails to be surjective.

Intersection theory on hyperquot schemes was used to obtain a quantum Schubert calculus on Grassmannians and flag varieties \cite{bertram,cf-flags,cf-partial,chen}.  These articles rely on the fact that Gromov-Witten invariants can be computed as intersection numbers on quot schemes.  The fundamental fact we use in the proof of Theorem~\ref{t:main} is an equivariant version of that statement:
\begin{proposition}
\label{p:product-quot}
The equivariant quantum product can be computed on (hyper)quot schemes.  That is, 
given permutations $v_1,\ldots,v_r$ in $\Sn$, the EQLR coefficient $c_{v_1,\ldots,v_r}^{w,\ul{d}}(t)$ is equal to 
\[
  \qpi^T_*( \mu_{v_1}\cdots\mu_{v_r}\cdot\tilde\mu_{w^\vee} ),
\]
where $\qpi^T_*\colon H_T^*(\hq_{\ul{d}}) \to H_T^*(\pt)$ is the equivariant pushforward to a point. 
\end{proposition}

To prove this, we will check equality of the coefficient of each monomial $(-\alpha_1)^{j_1}\cdots (-\alpha_{n-1})^{j_{n-1}}\cdot (-t_n)^{j_n}$, using Lemma~\ref{l:coefficient}.  In fact, we will see that these coefficients count points in certain intersections taking place inside the mixing space $\Mm(\dd)$, so they are nonnegative integers; see Corollary~\ref{c:eqlr-enum}.

We also write
\[
  ( \mu_{v_1}\cdots\mu_{v_r}\cdot\tilde\mu_{w^\vee} )_\dd^T = \qpi^T_*( \mu_{v_1}\cdots\mu_{v_r}\cdot\tilde\mu_{w^\vee} )
\]
for the EQLR coefficient $c_{v_1,\ldots,v_r}^{w,\ul{d}}(t)$. With this notation, we have
\[  \scl_{v_1} \qtp \scl_{v_2} \qtp \cdots \qtp \scl_{v_r} = \sum_{w,\ul{d}} \qq^{\dd} \, ( \mu_{v_1}\cdots\mu_{v_r}\cdot\tilde\mu_{w^\vee} )_\dd^T  \, \scl_w.
\]
By linearity of equivariant quantum cohomology of flag varieties and of equivariant cohomology of quot schemes, for any polynomial $F$ in variables indexed by $\Sn$, we obtain
\begin{corollary}
\label{c:poly-quot}
\[
F(\sigma,t) =  \sum_{w,\ul{d}} \qq^{\dd} \, (F(\mu,t) \cdot\tilde\mu_{w^\vee} )_\dd^T  \, \scl_w \text{ in } QH_T^*(Fl(\nn)).
\]
\end{corollary}

Unlike the Kontsevich compactification, there is no globally defined evaluation map from $\hq_\dd$ to $\Fl(\nn)$.  However, the boundary $\hq_\dd\setminus\mor(\dd)$ can be broken into pieces which do map to (different) partial flag varieties.  This is described in detail in \cite{bertram}, \cite{cf-flags}, and \cite{cf-partial}; we summarize the relevant facts here.

Fix $\nn$.  Given $\dd$, let $\ee = (e_1,\ldots,e_m)$ be such that
\begin{equation}\label{e-conds}
\begin{array}{rclcl}
  e_i &\leq& \min(n_i,d_i) &\quad \text{ for } &1\leq i\leq m, \text{ and }\\
  e_i-e_{i-1} &\leq& n_i-n_{i-1} &\quad \text{ for } &2\leq i\leq m.
\end{array}
\end{equation}
In addition to $\hq_\dd$ and $\Fl(\nn)$, we will also consider quot schemes $\hq_{\dd-\ee}$ (parametrizing quotients whose ranks are still $\nn$, but whose degrees are $\dd-\ee$) and partial flag varieties $\Fl(\nn')$ (where $\nn' = \{ n_1-e_1 \leq n_2-e_2\leq \cdots \leq n_m-e_m \}$).

\begin{theorem}[{\cite{cf-flags,cf-partial}}]\label{t:strat}
Assume some $e_i>0$.  There are smooth irreducible varieties $\uu_\ee$ with the following properties.

\begin{enumerate}
\item There is a morphism $h_\ee: \uu_\ee \to \hq_\dd$, which is birational onto its image.  Every point of $\hq_\dd\setminus\mor(\dd)$ lies in the image of $h_\ee$, for some $\ee$.

\medskip

\item There is a smooth morphism $\rho: \uu_\ee \to \PP^1 \times \hq_{\dd-\ee}$, whose image contains $\PP^1\times \mor(\dd-\ee)$.

\medskip

\item Fix a point $z\in \PP^1$, and write $\uu_\ee(z) = \rho^{-1}(\{z\}\times\hq_{\dd-\ee})$.  There is a natural morphism
\[
  \psi_\ee(z): \uu_\ee(z) \to \Fl(\nn'),
\]
and for each $w \in \Sn$, there is a $w' \in S^{\nn'}$ such that
\[
  h_\ee^{-1}(\D_w(z)) = \rho^{-1}(\PP^1 \times \D_w^{\ee}(z)) \cup \psi_\ee(z)^{-1}(\Omega_{w'}),
\]
where the superscript in $\D_w^{\ee}$ indicates the degeneracy locus inside $\hq_{\dd-\ee}$.  
The same holds with $\D_w(z)$ and $\Omega_{w'}$ replaced by $\tilde\D_{w}(z)$ and $\tilde\Omega_{w'}$, respectively. \label{t:strat3}
\end{enumerate}
Moreover, the morphisms $h_\ee$, $\rho$, and $\psi_\ee(z)$ are equivariant for  natural actions of $GL_n=GL(V)$.
\end{theorem}

\section{The mixing group}\label{s:mixing}

Recall that the $T$-equivariant cohomology of $\Fl(\nn)$ is computed as the ordinary cohomology of a flag bundle $\FFl=\FFl(\nn;E) \to \BB$, where $E=L_1\oplus\cdots\oplus L_n$.  There is a large group acting on $\FFl$, using the transitive automorphism group of $\BB$ together with a ``fiberwise'' action of lower-triangular matrices.  This group was introduced in \cite{anderson} and dubbed the ``mixing group'' in \cite{agm}; in this section, we describe its construction concretely for the present context. 

We identify $T$ with $(\CC^*)^n$ using the basis $\alpha_1,\, \ldots,\, \alpha_{n-1},\, t_n$ for $M$, where $\alpha_i = t_i-t_{i+1}$.  Using the approximation space $\BB = (\PP^{N-1})^n$ (for $N\gg 0$, as in \S\ref{ss:eqcoh}), set $M_i = pr_i^*\OO(-1)$, so $c_1(M_i) = \alpha_i$ for $1\leq i\leq n-1$, and $c_1(M_n) = t_n$.  With this setup, we have $L_i = M_i \otimes M_{i+1} \otimes \cdots \otimes M_n$, so $c_1(L_i) = t_i$.  

Observe that for $i\geq j$, the line bundle
\[
  L_i \otimes L_j^{-1} = M_j^{-1} \otimes \cdots \otimes M_{i-1}^{-1}
\]
is generated by global sections.  It follows that the bundle
\[
  End(E) = \bigoplus_{i,j} L_i \otimes L_j^{-1}
\]
has global sections in lower-triangular matrices.

Let $Aut(E) \subset End(E)$ be the automorphism bundle of $E$; this is a group scheme over $\BB$, whose fiber at $x$ is $GL(E(x))$.  The group $\Gamma_0 = \Hom_\BB(\BB,Aut(E))$ of global sections is a connected algebraic group over $\CC$.  This group acts on the total space of the bundle $E$, and hence also on the flag bundle $\FFl$, preserving the fibers of the projection to $\BB$.  The group acts similarly on $\Qq_\dd$, the mixing space for the quot scheme $\hq_\dd$.

The group $(PGL_N)^n$ acts naturally on $\Gamma_0$ via its transitive action on $\BB$.  Let $\Gamma = \Gamma_0\rtimes (PGL_N)^n$ be the semidirect product; this is a connected linear algebraic group acting on $\FFl$ and $\Qq_\dd$, which we call the \emph{mixing group}.

Let $\tilde{E}_\bullet$ be the flag of bundles on $\BB$ with $\tilde{E}_i = L_n\oplus \cdots \oplus L_{n+1-i}$, and let $\tilde{\OOmega}_w \subseteq \FFl$ be the corresponding Schubert loci.  Then $\tilde{\OOmega}_w \isom \EE \times^T \tilde\Omega_w$.

\begin{lemma}[{\cite{anderson}, \cite[\S6]{agm}}] \label{l:gp-orbits}
Let $X$ be a scheme with an action of $\tilde{B}$, and let $X_i$ be the $\tilde{B}$-orbits.  Let $\Xx = \EE \times^T X$ and $\Xx_i = \EE \times^T X_i$ be the mixing spaces.  Then the orbits of $\Gamma$ on $\Xx$ are $\Xx_i$.

In particular, the orbits of $\Gamma$ on $\FFl$ are the cells $\tilde\OOmega^\circ_w \isom \EE \times^T \tilde\Omega^\circ_w$ (so $\Gamma$ acts with finitely many orbits).
\end{lemma}

\section{Equivariant moving lemmas}\label{s:moving}

In this section we use the action of the mixing group to put certain subvarieties of the space of maps into (generically) transverse position.  Our goal is to prove Proposition~\ref{p:product-quot}, establishing that the EQLR coefficients can be computed on quot schemes.  Along the way, we will obtain an ``enumerative'' interpretation of these coefficients.

We will need a generalization of Kleiman's transversality theorem \cite[Theorem~8]{kleiman}.  As a matter of notation, for an algebraic group $G$ acting on a scheme $X$, a map $\psi\colon Z \to X$, and an element $g\in G$, we write $g\cdot Z$ to denote $Z$ equipped with the composite map $Z \xrightarrow{\psi} X \xrightarrow{g\cdot} X$.  (When $\psi$ is an embedding, this is simply the translation of $Z$ by $g$.)

A map $\psi\colon Z \to X$ is \define{dimensionally transverse} to a (locally closed) subvariety $V \subseteq X$ if $\codim_Z(\psi^{-1}V)=\codim_X(V)$.

\begin{proposition}\label{p:transverse}
Let $X$ be a smooth variety with an action of an algebraic group $G$, and let $Y$ and $Z$ be Cohen-Macaulay varieties with maps $\phi\colon Y \to X$ and $\psi\colon Z \to X$.  Assume that $\psi$ is dimensionally transverse to the orbits of $G$ on $X$, that is,
\[
  \codim_Z(\psi^{-1}O) = \codim_X(O)
\]
for all orbits $O$.  Then for general $g$ in $G$, the scheme $W_g = Y\times_X (g\cdot Z) $ is pure-dimensional, of dimension
\[
  \dim W_g = \dim Y + \dim Z - \dim X.
\]
(When $\psi$ is an embedding of $Z$ in $X$, this says $\codim_Y(\phi^{-1}(g\cdot Z)) = \codim_X(Z)$.)  If the ground field has characteristic zero, and $Y$ and $Z$ are reduced, then for general $g$, $W_g$ is also reduced.
\end{proposition}

The proof is the same as those appearing in \cite{kleiman} and \cite{speiser}: the essential point is that the action map $G\times Z \to X$ is flat.

We will apply Proposition~\ref{p:transverse} repeatedly in the situation where $X=\FFl$ is the flag bundle (or approximation space), or a product of such spaces.  In what follows, we work with the setup:
\begin{itemize}
\item $\Omega_1,\ldots,\Omega_r \subseteq \Fl$ are Schubert varieties of codimensions $\ell_1,\ldots,\ell_r$, and $\tilde\Omega$ is an opposite Schubert variety, of codimension $\ell$.

\item $p_1,\ldots,p_r,p$ are fixed, distinct points of $\PP^1$.

\item $J=(j_1,\ldots,j_n)\in \{0,\ldots,m-1\}^n$ is an index.

\item $A = \ell_1+\cdots+\ell_r+\ell$.
\end{itemize}

As usual, let $M_{r}(\dd)$ denote the space of maps with $r$ marked points, let $\Mbar_{r}(\dd)$ be the Kontsevich space of stable maps, and let $\Mm_{r}(\dd)$ and $\Mmbar_{r}(\dd)$ be the corresponding approximate mixing spaces.

\begin{lemma}[Moving Lemma A]\label{l:move-mbar}
There are elements $\gamma_1,\ldots,\gamma_r$ in the mixing group $\Gamma$ such that
\[
 \bar{W}' = \ev^{-1}(\gamma_1\cdot\OOmega_1 \times \cdots \times \gamma_r\cdot\OOmega_r \times \tilde\OOmega) \cap (\Mmbar_{r+1}(\dd))_J
\]
is either empty, or reduced and pure dimensional of codimension $A$ inside $(\Mmbar_{0,r+1}(\dd))_J$.  Moreover,
\[
 W' = \ev^{-1}(\gamma_1\cdot\OOmega_1 \times \cdots \times \gamma_r\cdot\OOmega_r \times \tilde\OOmega) \cap (\Mm_{r+1}(\dd))_J 
\]
is Zariski-dense inside $\bar{W}'$.

In fact, $\gamma_1,\ldots,\gamma_r$ can be chosen so that the same results hold for
\[
 \bar{W} = \ev^{-1}(\gamma_1\cdot\OOmega_1 \times \cdots \times \gamma_r\cdot\OOmega_r \times \tilde\OOmega) \cap (\Mmbar_{r+1}(\dd))_J \cap \tilde{f}^{-1}(x)
\]
and
\[
 W = \ev^{-1}(\gamma_1\cdot\OOmega_1 \times \cdots \times \gamma_r\cdot\OOmega_r \times \tilde\OOmega) \cap (\Mm_{r+1}(\dd))_J \cap \tilde{f}^{-1}(x),
\]
where $x \in \Mbar_{0,r+1}$ is the point corresponding to $(p_1,\ldots,p_r,p)$, and $\tilde{f}$ is the composition $\Mmbar_{r+1}(\dd) \to \Mmbar_{0,r+1}=\BB \times \Mbar_{0,r+1} \to \Mbar_{0,r+1}$.
\end{lemma}

\begin{proof}
Apply Proposition~\ref{p:transverse}, using the group $G = \Gamma^{(r+1)} = (\Gamma_0)^{r+1}\rtimes (PGL_N)^n$, with
\begin{align*}
 X &= \FFl\times_\BB \cdots \times_\BB \FFl  \quad (r+1 \text{ factors}), \\
 Z &= \OOmega_1\times_\BB \cdots \times_\BB \OOmega_r \times_\BB \FFl \hookrightarrow X, \\
\intertext{and}
 Y &= \ev^{-1}(\FFl\times_\BB\cdots \times_\BB \FFl\times_\BB \tilde\OOmega) \cap (\Mmbar_{r+1}(\dd))_J \xrightarrow{\ev|_Y} X.
\end{align*}
Observe that $Y$ has codimension $\ell$ in $(\Mmbar_{r+1}(\dd))_J$, and is therefore Cohen-Macaulay, since $\Mmbar_{r+1}(\dd)$ is.

The statement about $W'$ follows by showing that
\[
  \bar{W}'\setminus W' = \partial \bar{W}' = \ev^{-1}(\gamma_1\cdot\OOmega_1 \times \cdots \times \gamma_r\cdot\OOmega_r \times \tilde\OOmega) \cap (\partial\Mmbar_{r+1}(\dd))_J
\]
has strictly smaller dimension than $\bar{W}'$, where $\partial\Mmbar_{r+1}(\dd) = \Mmbar_{r+1}(\dd)\setminus \Mm_{r+1}(\dd)$ is the boundary divisor.  For this, apply the proposition with $Y$ replaced by $\partial Y$, defined similarly.  (This divisor in $Y$ is again Cohen-Macaulay, for the same reasons.)
\end{proof}

In Moving Lemma A, $W$ can be identified with $\gamma_1\cdot\OOmega_1(p_1) \cap \cdots \cap \gamma_r\cdot\OOmega_r(p_r) \cap {\tilde\OOmega}(p) \cap (\Mm(\dd))_J$, where $M(\dd)$ is the space of degree $\dd$ maps $\PP^1 \to \Fl$, and $\Omega_i(p_i)$ is the locus sending $p_i$ into $\Omega_i$.  This gives us our enumerative interpretation of the EQLR coefficients.

\begin{corollary}\label{c:eqlr-enum}
Write the polynomial $c_{v_1,\ldots,v_r}^{w,\dd}(t)$ as $\sum_J c_{v_1,\ldots,v_r}^{w,\dd,J}\, (-\alpha)^J$.  Then for general $\gamma_1,\ldots,\gamma_k$ chosen as in Lemma~\ref{l:move-mbar}, we have
\[
  c_{v_1,\ldots,v_r}^{w,\dd,J} = \begin{cases} \#W & \text{if } \dim W = 0; \\ 0 & \text{otherwise.}\end{cases}
\]
Consequently, the EQLR coefficient $c_{v_1,\ldots,v_r}^{w,\dd}(t)$ is \emph{Graham-positive}; that is, when written as a polynomial in the variables $-\alpha_i$, it has nonnegative coefficients.
\end{corollary}

\begin{proof}
Use the description of $c_{v_1,\ldots,v_r}^{w,\dd}(t)$ given in \eqref{e:eqlr}, together with the last part of Moving Lemma A.  Recalling that $f:\Mbar_{r+1}(\dd) \to \Mbar_{0,r+1}$ denotes the forgetful map, note that $[\tilde{f}^{-1}(x)] = f^*[\pt]$ as classes in $H^*\Mmbar_{r+1}(\dd) = H_T^*\Mbar_{r+1}(\dd)$.
\end{proof}

Next we must prove an analogue of Moving Lemma A for quot schemes.  Continuing the notation, we write $\D_i \subseteq \PP^1 \times \hq_\dd$ for the degeneracy locus of \eqref{e:dw} corresponding to the Schubert variety $\Omega_i$, and we write $\Qq_\dd$ and $\DD_i$ for the mixing spaces.

\begin{lemma}[Moving Lemma B]\label{l:move-quot}
There are elements $\gamma_1,\ldots,\gamma_r$ in the mixing group $\Gamma$ such that
\[
 \bar{W}= \gamma_1\cdot\DD_1(p_1) \cap \cdots \cap \gamma_r\cdot\DD_r(p_r) \cap {\tilde\DD}(p) \cap (\Qq_\dd)_J
\]
is either empty, or reduced and pure dimensional of codimension $A$ inside $(\Qq_\dd)_J$.  Moreover,
\[
 W = \gamma_1\cdot\OOmega_1(p_1) \cap \cdots \cap \gamma_r\cdot\OOmega_r(p_r) \cap {\tilde\OOmega}(p) \cap (\Mm(\dd))_J
\]
is Zariski-dense inside $\bar{W}$.
\end{lemma}

The proof is essentially the same as that of \cite[Theorem~4.3(ii)]{cf-partial}, using the mixing group action and Proposition~\ref{p:transverse} in place of the transitive group action on $\Fl(\nn)$.

\begin{proof}
In Moving Lemma A, we have already seen that $W$ is reduced and pure-dimensional of codimension $A$ (when nonempty).  To prove this lemma, it will suffice to show that $\bar{W} \cap (\partial\Qq_\dd)_J$ has codimension greater than $A$ in $(\Qq_\dd)_J$, where $\partial\Qq_\dd = \Qq_\dd\setminus\Mm(\dd)$ is the boundary.

Consider $\ee$ as in \eqref{e-conds}, with some $e_i>0$.  By induction on $\dd$, we may assume the statement holds for $\Qq_{\dd-\ee}$.  (The base case is $\Qq_\mathbf{0} = \FFl$, where Proposition~\ref{p:transverse} applies directly.)  Let $\Uu_\ee$ be the mixing space for the variety $\uu_\ee$ of Theorem~\ref{t:strat}.  Since the map $h_\ee:\Uu_\ee \to \Qq_\dd$ is birational onto its image, and $\partial\Qq_\dd$ is covered by such images, it will suffice to show
\begin{equation}\label{e:ineq}
  \codim_{\Uu_\ee}\left( \bigcap_{i=1}^r h_\ee^{-1}(\gamma_i\cdot\DD_i(p_i)) \cap h_\ee^{-1}(\tilde\DD(p)) \right) > A - \codim_{\Qq_\dd}(h_\ee(\Uu_\ee))
\end{equation}
for general $\gamma_1,\ldots,\gamma_r$ in $\Gamma$.  Using Theorem~\ref{t:strat}\eqref{t:strat3}, the intersection on the LHS can be written as
\begin{equation}
\begin{array}{c}
  \displaystyle{\bigcap_{i=1}^r \left( \rho^{-1}(\PP^1 \times \gamma_i\cdot\DD^{\ee}_i(p_i)) \cup \psi_\ee(p_i)^{-1}( \gamma_i\cdot \OOmega_{w'_i} ) \right)} \qquad \qquad \\
  \qquad \qquad \qquad \displaystyle{\cap \left( \rho^{-1}(\PP^1 \times \tilde\DD^{\ee}(p)) \cup \psi_\ee(p)^{-1}(\tilde\OOmega_{w'}) \right) }.
\end{array}
\end{equation}
(Here $\Omega_{w'_i} \subseteq \Fl(\nn')$ is the Schubert variety whose existence is claimed in Theorem~\ref{t:strat}, where $\Omega_i = \Omega_{w_i}$.)  Since the points $p_1,\ldots,p_r,p$ are distinct, any nonempty component of the expansion of this intersection contains at most one factor of $\psi_\ee(p_i)^{-1}(\OOmega_{w'_i})$ or $\psi_\ee(p)^{-1}(\tilde\OOmega_{w'})$.

If there are no factors of $\psi_\ee(p_i)^{-1}(\OOmega_{w'_i})$ or $\psi_\ee(p)^{-1}(\tilde\OOmega_{w'})$, the intersection is the inverse image of a similar intersection on $\Qq_{\dd-\ee}$.  By induction and the fact that $\rho$ is smooth, the inequality \eqref{e:ineq} holds for these components.

For the case where $\psi_\ee(p)^{-1}(\OOmega_{w_r'})$ occurs, let
\begin{align*}
 X &= \FFl(\nn'), \\
 Z &= \OOmega_{w_r'} \hookrightarrow X, \\
\intertext{and}
 Y &= \bigcap_{i=1}^{r-1} \rho^{-1}(\{p_r\}\times \gamma_i\cdot\DD_i(p_i)) \cap \rho^{-1}(\{p_r\} \times \tilde\DD(p)) \xrightarrow{\psi_\ee(p_r)} X.
\end{align*}
By the previous case, $Y$ has codimension $\ell_1+\cdots+\ell_{r-1}+\ell$ in $\Uu_\ee(p_r) = \rho^{-1}(\{p_r\}\times\Qq_{\dd-\ee})$.  By Proposition~\ref{p:transverse}, we can choose $\gamma_r\in\Gamma$ so that
\[
  Y \cap \psi_\ee(p_r)^{-1}(\gamma_r\cdot\OOmega_{w'_r})
\]
has codimension $\ell(w'_r)$ in $Y$.  The inequality \eqref{e:ineq} follows from the estimates on $\ell(w'_r)$ from \cite[Lemma~5.5(ii), Lemma~5.8(i)]{cf-partial}, as in \cite[\S6.1]{cf-partial}.

Finally, for the case where $\psi_\ee(p)^{-1}(\tilde\OOmega_{w'})$ occurs, first observe that each $D^\ee_i(p_i) \subseteq \hq_{\dd-\ee}$ is a $B$-invariant subscheme (where $B\subseteq GL_n$ is upper-triangular matrices), and it is dimensionally transverse to all $\tilde{B}$-orbits in $\hq_{\dd-\ee}$.  
It follows that $\DD^\ee_i(p_i)$ is dimensionally transverse to $\Gamma$-orbits on $\Qq_{\dd-\ee}$.  Now apply Proposition~\ref{p:transverse}, with
\begin{align*}
 X &= \Qq_{\dd-\ee}\times_\BB \cdots \times_\BB \Qq_{\dd-\ee}  \quad (r \text{ factors}), \\
 Z &= \DD^\ee_1(p_1) \times_\BB \cdots \times_\BB \DD^\ee_r(p_r) \hookrightarrow X, \\
\intertext{and}
 Y &= \psi_\ee(p)^{-1}(\tilde\OOmega_{w'}) \xrightarrow{\bar\rho \times \cdots \times \bar\rho} X,
\end{align*}
where $\bar\rho$ is the restriction of $\rho:\Uu_\ee \to \PP^1 \times \Qq_{\dd-\ee}$ to $\psi_\ee(p)^{-1}(\tilde\OOmega_{w'}) \subseteq \Uu_\ee(p) = \rho^{-1}(\{p\}\times \Qq_{\dd-\ee})$.  We find $\gamma=(\gamma_1,\ldots,\gamma_r)$ in $\Gamma^{(r)}$ such that
\[
  (\bar\rho\times\cdots\times\bar\rho)^{-1}(\gamma\cdot Z) = \bigcap_{i=1}^r \rho^{-1}(\PP^1 \times \gamma_i\cdot \DD^\ee_i(p_i)) \cap \psi_\ee(p)^{-1}(\tilde\OOmega_{w'})
\]
has codimension $\ell_1+\cdots+\ell_r$ in $Y$.  Since $Y$ has codimension $\ell(w')$ in $\Uu_\ee(p)$, the same estimate as in the previous case proves the inequality \eqref{e:ineq}.
\end{proof}

We can now complete the proof of Proposition~\ref{p:product-quot}.

\begin{corollary}\label{c:product-quot}
Writing $c_{v_1,\ldots,v_r}^{w,\dd}(t) = \sum_J c_{v_1,\ldots,v_r}^{w,\dd,J}\, (-\alpha)^J$ as above, we have
\[
  c_{v_1,\ldots,v_r}^{w,\dd,J} = \pi^{\Qq_\dd}_*( [\DD_1(p_1)]\cdots[\DD_r(p_r)]\cdot[{\tilde\DD}(p)]\cdot[(\Qq_d)_J] ),
\]
where $\pi^{\Qq_\dd}$ is the map $\Qq_\dd \to \pt$.
\end{corollary}

\begin{proof}
By Moving Lemma B, the right-hand side is equal to $\#W$.  However, this $W$ is the same as the one from Moving Lemma A, where $x$ corresponds to the point $(p_1,\ldots,p_r,p) \in \Mbar_{0,r+1}$, so the claim follows from Corollary~\ref{c:eqlr-enum}.  
\end{proof}

\section{Proofs of the main results}\label{s:proofs}

For a partial flag variety $\Fl(\nn)$, recall that the equivariant quantum Schubert polynomial $\Sch^q_w(x,t)$ is defined to be polynomial obtained from the universal double Schubert polynomial $\Sch_w(g,h)$ (or $\Sch_w(g,y)$) by the specialization \eqref{e:spec1}, setting the variables $g_i[0]$ to $x_i$, the variables $h_i[0]$ (or $y_i$) to the variables $t_i$, the variables $g_{n_{i-1}+1}[n_{i+1}-n_{i-1}-1]$ to $(-1)^{n_i-n_{i-1}+1}q_i$ for $1\leq i\leq m$, and setting all other $g_i[j]$ and $h_i[j]$ to zero.  Letting $\sigma_i^j$ be the $i$th symmetric polynomial in $x_{n_{j-1}+1},\ldots,x_{n_j}$, we write $\Sch^q_w(x,t)$ as a polynomial in the $\sigma_i^j$ and $t_i$, and denote this by $\Sch^q_w(\sigma,t)$.

In \eqref{e:q-vars}, we observed that $g_{n_{i-1}+1}[n_{i+1}-n_{i-1}-1]$ maps to the class
\[
   Q_{n_{i-1}+1}[n_{i+1}-n_{i-1}-1] = c^T_{n_{i+1}-n_{i-1}}(\ker(\Aa^*_{i+1}\to \Aa^*_i))
\]
in $H_T^*\hq_{\ul{d}}$.  Note that the same variable maps to $(-1)^{n_i-n_{i-1}+1}q_i$.  In the case of the complete flag variety, the class $Q_i[1]$ represents the locus on $\hq_\dd$ where $\Aa^*_{i+1} \to \Aa^*_i$ is not surjective, and the corresponding variable $g_i[1]$ specializes to $q_i$.

More generally, for any polynomial $P(g,y)$ in variables $g_i[j]$ and $y_i$, denote by $P^q(x,t)$ the polynomial which results after performing the above substitutions.  Denote by $e^q_k(l)$ the polynomial obtained by the above substitutions into $E_k^l(g)$, for $0\leq k\leq l$ (see \eqref{e:Ekl-inductive}).  Since $e^q_k(n)$ is symmetric in $x_{n_{j-1}+1},\ldots,x_{n_j}$ for every $1\leq j\leq m+1$, it can be written as a polynomial in the $\sigma_i^j$, which we denote by $\tilde{e}_k^q(m+1)$.

Our main goal is to prove the equivariant quantum Giambelli formula, which we restate here for convenience:

\begin{theorem} \label{t:sigma-usp}
For $w\in\Sn$, we have $\sigma_w = \Sch^q_w(\sigma,t)$ in $QH_T^*(\flag)$.
\end{theorem}
 
In the course of proving this, we will simultaneously prove an auxiliary result, as in \cite{chen}.

\begin{proposition}\label{p:poly-in-Q}
Let $P(Q,t)$ be a polynomial in the classes $Q_i[j]$ and $t$, where $i+j\in\nn$.  Then
\[
 P^q(x,t) = \sum_{\ul{d},v} q^{\ul{d}}\, (P(Q,t)\cdot\tilde{\mu}_{v^\vee})^T_{\ul{d}} \,\sigma_v \text{ in }QH_T^*(\flag).
\]
\end{proposition}


Before proceeding to the proofs, we note that the presentations of the cohomology ring of $\Fl(\nn)$ given in \cite{kim-eq} can be deduced from  Proposition~\ref{p:poly-in-Q}.  Recall from \S\ref{ss:eqcoh} that $\Lambda_T = H_T^*(\pt) \isom \ZZ[t_1,\ldots,t_n]$, and let $q$ and $x$ stand for the variable sets $(q_1,\ldots,q_m)$ and $(x_1,\ldots,x_n)$, respectively.

\begin{corollary}\label{c:presentation}
We have
\[
 QH_T^*(\Fl(\nn)) \isom \Lambda_T[q][\sigma_1^1,\ldots,\sigma_{n_1}^1,\ldots,\sigma_{1}^{{m+1}},\ldots,\sigma_{n-n_m}^{{m+1}}]/I_T^q
\]
where $I^q_T$ is the ideal $(\tilde{e}_1^q(m+1)-e_1(t),\ldots,\tilde{e}_n^q(m+1)-e_n(t))$.  In the special case of the complete flag variety $\Fl(\CC^n)$, this gives
\[
 QH_T^*(\Fl(\CC^n)) \isom \Lambda_T[q][x]/(e_1^q(n)-e_1(t),\ldots,e_n^q(n)-e_n(t)).
\]
\end{corollary}

\begin{proof}
The argument is the same as in \cite[Theorem~7.1]{chen}; we sketch it here.  We have $E_k^n(Q) = c_k^T(A^*_n) = c_k^T(V^*) = e_k(t_1,\ldots,t_n)$, since $V^*$ is a trivial (but not equivariantly trivial) vector bundle.  Now apply Proposition~\ref{p:poly-in-Q} to the polynomials $P(Q,t) = E_k^n(Q)-e_k(t) = 0$, for $k = 1,\ldots,n$, obtaining $n$ relations which specialize to the known relations defining $H_T^*(\Fl(\nn))$.  The claim follows.
\end{proof}

In the proof of Theorem~\ref{t:sigma-usp} and Proposition~\ref{p:poly-in-Q}, we will use three lemmas.

\begin{lemma}
\label{l:sigma-id}
For $w\in\Sn$, $\sigma_w = \sum_{\ul{d},v} q^{\ul{d}}\,(\mu_w\cdot\tilde{\mu}_{v^\vee})^T_{\ul{d}}\,\sigma_v$ in $QH_T^*(\flag)$.
\end{lemma}

%
%

\begin{proof} 
When $\ul{d}=0$, the quot scheme $\quot_{\dd}$ is the flag variety $\flag$, and in this case duality (see (\ref{e:eq-duality})) gives $(\mu_w\cdot\tilde{\mu}_{v^\vee})^T_{\ul{0}}= [\Omega_w]^T\cdot[\tilde{\Omega}_{v^\vee}]^T=\delta_{wv}$.

More generally, we have
\[
  (\mu_w\cdot\tilde{\mu}_{v^\vee})^T_{\dd} = \qpi^T_*(\mu_w \cdot \tilde\mu_{v^\vee}) = \qpi^T_*([D_w(0)]\cdot[\tilde{D}_{v^\vee}(\infty)]).
\]
For degree reasons, $(\mu_w\cdot\tilde{\mu}_{v^\vee})^T_{\dd}$ can be nonzero only when $\ell(w)+\ell(v^\vee) \geq \dim \quot_{\dd}$.  
Using Lemma~\ref{l:move-quot}, in the case $r=1$, one sees that $D_w(0)$ and $\tilde{D}_{v^\vee}(\infty)$ are dimensionally transverse.  It follows that
\[
  [D_w(0)]\cdot[\tilde{D}_{v^\vee}(\infty)] = [D_w(0)\cap \tilde{D}_{v^\vee}(\infty)],
\]
and that $D_w(0)\cap \tilde{D}_{v^\vee}(\infty)$ is empty when $\ell(w)+\ell(v^\vee) > \dim \quot_{\dd}$.  
Therefore $(\mu_w\cdot\tilde{\mu}_{v^\vee})^T_{\dd}$ can be nonzero only when $\ell(w)+\ell(v^\vee) = \dim \quot_{\dd}$; that is, when $D_w(0)\cap \tilde{D}_{v^\vee}(\infty)$ consists of finitely many points.  Moreover, Lemma~\ref{l:move-quot} also implies $D_w(0)\cap \tilde{D}_{v^\vee}(\infty) \cap \mor(\dd)$ is dense in $D_w(0)\cap \tilde{D}_{v^\vee}(\infty)$, so these points must lie in $\mor(\dd)$.

Finally, recall there is an action of $\CC^*$ on $\quot_\dd$, coming from the standard action on $\PP^1$, and $D_w(0)\cap \tilde{D}_{v^\vee}(\infty)$ is stable under this $\CC^*$ action.  When $\dd\neq 0$, this action has no fixed points inside $\mor(\dd)$, since a non-constant map is changed by reparametrization.  Therefore $D_w(0)\cap \tilde{D}_{v^\vee}(\infty)$ is empty whenever $\ell(w)+\ell(v^\vee) = \dim \quot_{\dd}$, and it follows that $(\mu_w\cdot\tilde{\mu}_{v^\vee})^T_{\dd}=0$ in this case.
\end{proof}

The following lemma shows that a similar sum which involves a single $Q_i[j]$ only has one non-zero  term. Let $\ul{e}_l=(0,\ldots,1,\ldots,0)$ be the $m$-tuple whose only nonzero entry is in the $l$th position. 
\begin{lemma}
\label{l:single-Q} 
For $1\leq l\leq m$ and $1\leq i\leq n_{l}$, 
 \[
\sum_{\ul{d},v} q^{\ul{d}} \,(Q_i[n_{l+1}-i]\cdot\tilde{\mu}_{v^\vee})^T_{\ul{d}}\, \sigma_v= (Q_{n_{l-1}+1}[n_{l+1}-n_{l-1}-1]\cdot\tilde{\mu}_{w^\circ})^T_{\ul{e}_l}
\] 
if $i= n_{l-1}+1$, and is equal to zero otherwise.


\end{lemma}
\begin{proof}
Note that when $d_a=0$, $A_a\rightarrow A_b$ is an inclusion of vector bundles on $\hq_\dd$ for all $b>a$, or equivalently $A_b^*\rightarrow A_a^*$ is a surjection of bundles on $\hq_\dd$.

Let $l'<l$ be such that $n_{l'}+1\leq i\leq n_{l'+1}$. If $d_a=0$ for any $l'+1\leq a< l$, then by Lemma \ref{l:Q-zero} applied to $b=l+1$, we obtain $Q_{i}[n_{l+1}-i]=0$, so that $Q_{i}[n_{l+1}-i]\neq 0$ only if $d_{l'+1},\ldots,d_{l-1}> 0$.

By this and degree considerations, we conclude that $(Q_i[n_{l+1}-i]\cdot\tilde{\mu}_{v^\vee})^T_{\ul{d}}$ is zero unless 
\begin{align*} 
n_{l+1}-i+1+\ell(v^\vee)& \geq \dim \quot_{\ul{d}} =  \dim F + \sum_{k=1}^m d_k(n_{k+1}-n_{k-1})\\
& \geq \dim F + n_{l+1}+n_l-n_{l'}-n_{l'+1}.
\end{align*}
Since $n_{l+1}-n_{l'}+\dim F \geq 
n_{l+1}-i+1+\ell(v^\vee)$, these inequalities can only hold when $l'=l-1$ and every inequality is an equality. Therefore $\ul{d}=\ul{e}_l$, $i=n_{l-1}+1$, and $v^\vee=w^\circ$, the permutation of longest length. 
\end{proof}

\begin{lemma}
\label{l:single-computation-Q}
Assume that Proposition~\ref{p:poly-in-Q} holds for polynomials
$P(Q,t)$ of $Q$-degree up to $n_{l+1}-n_{l-1}-1$.  Then 
\[
  (Q_{n_{l-1}+1}[n_{l+1}-n_{l-1}-1]\cdot\tilde{\mu}_{w^\circ})^T_{\ul{e}_l} = (-1)^{n_{l+1}-n_l+1}.
\]
\end{lemma}

\begin{proof}
Recall that $( Q_{n_{l-1}+1}[n_{l+1}-n_{l-1}-1]\cdot\tilde{\mu}_{w^\circ} )^T_{\ul{e}_l}$ is defined as the equivariant pushforward $\qpi^T_*(Q_{n_{l-1}+1}[n_{l+1}-n_{l-1}-1]\cdot\tilde{\mu}_{w^\circ})$, for $\qpi^T_* : H_T^*\quot_{\ul{e}_l} \to H_T^*(\pt)$.  Moreover, since the degree of $Q_{n_{l-1}+1}[n_{l+1}-n_{l-1}-1]\cdot\tilde{\mu}_{w^\circ}$ is equal to $n_{l+1}-n_{l-1} + \ell(w^\circ) = \dim \quot_{\ul{e}_l}$, the equivariant pushforward is equal to the non-equivariant one: 
\[
  \qpi^T_*( Q_{n_{l-1}+1}[n_{l+1}-n_{l-1}-1]\cdot\tilde{\mu}_{w^\circ} ) = \qpi_*( \bar{Q}_{n_{l-1}+1}[n_{l+1}-n_{l-1}-1]\cdot \bar{\tilde{\mu}}_{w^\circ} )
\]
in $H_T^0(\pt) = H^0(\pt) = \ZZ$, where for a class $\gamma$ in $H_T^*\quot_{\ul{e}_l}$, we write $\bar\gamma$ for its image in $H^*\quot_{\ul{e}_l}$.  By \cite[Proposition~11.1]{chen}, the right-hand side is $(-1)^{n_{l+1}-n_{l}+1}$, as desired.  (To apply this result of \cite{chen}, note that our Proposition~\ref{p:poly-in-Q} for polynomials $P(Q,t)$ of $Q$-degree up to $n_{l+1}-n_{l-1}-1$ implies the corresponding non-equivariant statement in \cite[Proposition~8.1]{chen}.)
\end{proof}

We are now ready to prove the main results of this section.

\begin{proof}[Proof of Proposition \ref{p:poly-in-Q} and Theorem \ref{t:sigma-usp}] 
We will use induction on the length of $w$, and on the degree of $P(Q,t)$, viewed as a polynomial in $Q_i[j]$ with coefficients in $\ZZ[t]$.
  
We proceed by showing that Theorem~\ref{t:sigma-usp} for $\ell(w)\leq s$ implies Proposition~\ref{p:poly-in-Q} for polynomials $P(Q,t)$ of $Q$-degree at most $s$, and then showing that Proposition~\ref{p:poly-in-Q} for polynomials $P(Q,t)$ of $Q$-degree at most $s$, together with Lemma~\ref{l:single-Q}, imply Theorem~\ref{t:sigma-usp} for permutations $w$ of length at most $s+1$. The base case $s=0$ for Theorem~\ref{t:sigma-usp} holds because $\sigma_{id}=1$ as needed.  
 
First suppose that Theorem~\ref{t:sigma-usp} holds when $\ell(w)\leq s$.  
Let $P(Q,t)$ be a polynomial of $Q$-degree at most $s$. Then its terms involve only $t_i$ and $Q_i[j]$ with $j\leq s-1$ and $i+j\in\nn$.  By Lemma~\ref{l:rewrite-g}, each $Q_i[j]$ that could occur can be written as a polynomial in $t$ and $\Sch^\nn_w(Q,t)$ with $\ell(w)\leq s$.  Therefore a polynomial $P(Q,t)$ in $t$ and $Q_i[j]$ with $j\leq s-1$ can be rewritten as
\[
  P(Q,t)= F(\Sch^\nn_w(Q,t),t)=F(\mu_w,t),
\]
where the second equality holds by Proposition~\ref{p:deg-mu}.  Note that $F$ is a polynomial in $\mu_w$ and $t$ with $w\in \Sn$ and  $\ell(w)\leq s$.

By Corollary~\ref{c:poly-quot}, we obtain the equalities
\[
 \sum_{\ul{d},v} q^{\ul{d}}\,(P(Q,t))\cdot\tilde{\mu}_{v^\vee})^T_{\ul{d}}\, \sigma_v  = \sum_{\ul{d},v} q^{\ul{d}}\,(F(\mu,t)\cdot\tilde{\mu}_{v^\vee})^T_{\ul{d}} \,\sigma_v = F(\sigma,t)
\]
in $QH_T^*(Fl(\nn))$.  Since $F(\sigma,t)$ is a polynomial in $\sigma_w$ with $\ell(w)\leq s$, we can apply our hypothesis to obtain $F(\sigma,t) = F(\sigma_w^q(x,t),t) = P^q(x,t)$
so that Proposition~\ref{p:poly-in-Q} holds for polynomials $P(Q,t)$ with $Q$-degree at most $s$.

Now suppose that Proposition~\ref{p:poly-in-Q} holds for polynomials of $Q$-degree at most $s$.  Let $w\in\Sn$ with $\ell(w)=s+1$. Consider the polynomial $\Sch_w(Q,t)$.  By degree considerations, since the degree of $Q_i[j]$ is $j+1$, each $Q_i[s]$ can only appear linearly, and we can write $\Sch_w(Q,t) = P(Q,t) + \sum a_i Q_i[s]$, where $P(Q,t)$ is a polynomial of $Q$-degree at most $s$, and $a_i\in\ZZ$.  Therefore
\begin{align*}
\sigma_w &= \sum_{\ul{d},v} q^{\ul{d}}\,(\Sch_w(Q,t)\cdot\tilde{\mu}_{v^\vee})_{\ul{d}} \,\sigma_v \\
  &= \sum_{\ul{d},v} q^{\ul{d}}\, (P(Q,t))\cdot\tilde{\mu}_{v^\vee})_{\ul{d}} \,\sigma_v + \sum_i a_i\sum_{\ul{d},v} q^{\ul{d}}\,(Q_i[s]\cdot\tilde{\mu}_{v^\vee})_{\ul{d}}\, \sigma_v \\
  &= P^q(x,t) +  \sum a_i (-1)^{n_{l+1}-n_l+1}q_l \\
  &= \Sch_w^q(x,t),
\end{align*}
where the final sum is over $(i,s)=(n_{l-1}+1,n_{l+1}-n_{l-1}-1)$.  The last line follows from Proposition~\ref{p:poly-in-Q} for polynomials of $Q$-degree at most $s$, Lemma~\ref{l:single-Q}, and Lemma~\ref{l:single-computation-Q}, which holds since $s=n_{l+1}-n_{l-1}-1$.  We have shown that Theorem~\ref{t:sigma-usp} holds for permutations $w\in\Sn$ of length at most $s+1$.

This concludes the proof of Theorem~\ref{t:sigma-usp} and Proposition~\ref{p:poly-in-Q}, and therefore the proofs of the equivariant quantum Giambelli formula and presentation of the equivariant quantum cohomology ring for partial flag varieties.
\end{proof}

\section{Further properties}\label{s:stability}

We conclude with some brief remarks about the stability properties of the equivariant quantum Schubert polynomials.  First, the universal Schubert polynomials are independent of $n$, from the definition (see \S\ref{s:univ-sch} and \cite[\S2]{fulton}):
\begin{lemma}\label{l:univ-stable}
Consider $w\in S_n$.  The polynomial $\Sch_w(c,d)$ is the same when $w$ is considered as a permutation in $S_{n'}$, for $n'>n$, using the standard embedding of symmetric groups.
\end{lemma}

As an immediate consequence, we have a weak stability property of equivariant quantum Schubert polynomials for partial flags.  Let $\nn = (n_1<\cdots<n_m<n_{m+1}=n)$, and for any $n'>n$, let $\nn'=(n_1<\cdots<n_m<n_{m+1}<n_{m+2}=n')$.  The standard embedding $S_n \subset S_{n'}$ leads to a canonical inclusion $S^\nn \subset S^{\nn'}$.  (Concretely, given $w\in S^{\nn}$, the corresponding minimal length representative in $S^{\nn'}$ is given by appending $n_{m+1}+1,n_{m+1}+2,\ldots,n_{m+2}=n'$ to $w$.)  Then $\Sch^q_w(\sigma,t)$ is the same whether $w$ is considered in $S^{\nn}$ or $S^{\nn'}$.

For the remainder of this section, we will focus on the complete flag case.  Consider the standard embeddings of symmetric groups $S_n \subset S_{n+1} \subset \cdots \subset S_\infty$, and for each $n$, let $I^q_T(n)$ be the ideal $(e^q_1(n)-e_1(t),\ldots,e^q_n(n)-e_n(t))$ from Corollary~\ref{c:presentation}.  In analogy with \cite[Theorem~10.1]{fgp}, the polynomials for complete flags may be characterized as follows:

\begin{proposition}\label{p:stability}
For $w\in S_n$, the equivariant quantum Schubert polynomial $\Sch^q_w(x,t)$ is the unique polynomial in $\ZZ[t_1,\ldots,t_n;q_1,\ldots,q_{n-1};x_1,\ldots,x_n]$ with the property that, for all $N\geq n$, $\Sch^q_w(x,t)$ represents the Schubert class $\sigma_w$ in the ring
\[
 QH_T^*(\Fl(n)) = \ZZ[t_1,\ldots,t_N;q_1,\ldots,q_{N-1};x_1,\ldots,x_N] / I_T^q(N).
\]
\end{proposition}

\begin{proof}
As remarked above, in Lemma~\ref{l:univ-stable}, the polynomial $\Sch^q_w(x,t)$ is independent of $n$, so long as $w\in S_n$.
That $\Sch^q_w(x,t)$ represents $\sigma_w$ is the content of Theorem~\ref{t:main}, so the only question is uniqueness.  This follows from Lemma~\ref{l:basis} below.
\end{proof}

As polynomials in $x$ and $q$, our $e^q_k(l)$ are the same as the quantum elementary symmetric polynomials considered in \cite{fgp} (and denoted $E_k^l$ there); the following ``quantum straightening lemma'' therefore applies without change.
\begin{lemma}[{\cite[Lemma~3.5]{fgp}}]\label{l:straighten}
For $0\leq j,k\leq l$, we have
\begin{multline*}
  e_j^q(l)\, e_{k+1}^q(l+1) + e_{j+1}^q(l)\, e_k^q(l) + q_l\, e_{j-1}^q(l-1)\, e_k^q(l)\qquad \qquad \quad \\
 \qquad \qquad \qquad = e_k^q(l)\, e_{j+1}^q(l+1) + e_{k+1}^q(l)\, e_j^q(l) + q_l\, e_{k-1}^q(l-1)\, e_j^q(l).  \qquad \quad {\qed}
\end{multline*}
\end{lemma}

As in \cite{fgp}, this straightening relation means that any monomial in the quantum elementary symmetric polynomials can be written as a $\ZZ[q]$-linear combination of ``standard'' monomials $e^q_{i_1}(1)\cdot e^q_{i_2}(2)\cdots e^q_{i_n}(n)$, with $0\leq i_l\leq l$.

\begin{lemma}\label{l:basis}
For each $n$, the following $\ZZ[t,q]$-submodules of $\ZZ[t,q,x]$ are the same:
\begin{itemize}
\item the submodule spanned by equivariant quantum Schubert polynomials $\Sch^q_w(x,t)$, for $w\in S_n$;

\item the submodule spanned by ordinary (single) quantum Schubert polynomials $\Sch^q_w(x)$, for $w\in S_n$;

\item the submodule spanned by the monomials $x_1^{a_1}\cdots x_n^{a_n}$ with $a_i \leq n-i$.
\end{itemize}
Moreover, in each case the submodule is free and the indicated spanning set is a basis. 
\end{lemma}

\begin{proof}
The argument is the same as in \cite[\S3]{fgp}. 
From the definition of double universal Schubert polynomials \eqref{e:univ-double-def}, we have
\[
  \Sch^q_w(x,t) = \Sch_w^q(x) + R,
\]
where $R$ lies in the $\ZZ[t,q]$-submodule spanned by Schubert polynomials $\Sch^q_u(x)$ with $u<w$.  It follows that the equivariant quantum Schubert polynomials lie in the span of the ordinary quantum Schubert polynomials, and since the transition matrix between these two bases is uni-triangular, the submodules are the same.  The equality of the last two spans follows directly from \cite[Proposition~3.6]{fgp}.
\end{proof}

Taking $n$ to infinity, it follows from Lemma~\ref{l:basis} that the equivariant quantum Schubert polynomials $\Sch_w^q(x,t)$ form a $\ZZ[t,q]$-basis for $\ZZ[t,q,x]$, as $w$ runs over $S_\infty$, and all three sets of variables are infinite.  This means that any product $\Sch_u^q \cdot \Sch_v^q$ can be expanded as a $\ZZ[t,q]$-linear combination of equivariant quantum Schubert polynomials.  We conclude with an observation about these products.

%

It will be convenient to have notation for the polynomial ring: let
\[
  R_N = \ZZ[t_1,\ldots,t_N;q_1,\ldots,q_{N-1};x_1,\ldots,x_N].
\]

\begin{lemma}\label{l:ideal}
For $n<N$, consider the standard embedding $S_n \subset S_N$.  Let $w\in S_N$, so $\Sch_w^q=\Sch_w^q(x,t)$ is a polynomial in $R_N$.  Let $J$ be the ideal
\[
  J = I_T^q(n)\cdot R_N + (t_{n+1},\ldots,t_N)\cdot R_N + (q_n,\ldots,q_{N-1})\cdot R_N + (x_{n+1},\ldots,x_N)\cdot R_N.
\]
If $w\not\in S_n$, the polynomial $\Sch_w^q$ lies in $J$.  Equivalently, $\Sch_w^q$ maps to zero in $QH_T^*(\Fl(n)) \isom R_N/J$.
\end{lemma}

\begin{proof}
Consider a quot scheme $\quot_\dd$ compactifying maps to $\Fl(n)$.  For any $u \in S_N$ and $v \in S_n$, the proof of Lemma~\ref{l:sigma-id} shows that $(\mu_u\cdot \tilde\mu_{v^\vee})^T_\dd = 0$ unless $\dd=\mathbf{0}$ and $u=v$.  Since $w\not\in S_n$, we have
\[
 (\Sch_w(Q,t)\cdot \tilde\mu_{v^\vee})^T_\dd = (\mu_w\cdot \tilde\mu_{v^\vee})^T_\dd = 0
\]
for all $\dd$, where for $i,j >n$ and $k>n-1$, we set the extra variables $t_i=Q_j[0]=Q_k[1]=0$ in $\Sch_w(Q,t)$.  The lemma now follows from Proposition~\ref{p:poly-in-Q}.
\end{proof}

An immediate consequence is that the equivariant quantum Schubert polynomials multiply like Schubert classes in $QH_T^*(\Fl(n))$.

\begin{corollary}\label{c:multiplication}
Given permutations $u,v\in S_n$, expand the product of equivariant quantum Schubert polynomials as
\[
  \Sch^q_u(x,t) \cdot \Sch^q_v(x,t) = \sum_w a_w \, \Sch^q_w(x,t),
\]
with $w\in S_\infty$ and $a_w\in \ZZ[t,q]$.  Then the coefficient of $\qq^\dd$ in $a_w$ is equal to $c_{u,v}^{w,\dd}$ when $w\in S_n$.

In other words, the equivariant quantum product $\sigma_u \circ \sigma_v$ (in $QH_T^*(\Fl(n))$) is equal to the product of the polynomials $\Sch^q_u$ and $\Sch^q_v$, after discarding the terms $a_w \Sch^q_w$ for $w\not\in S_n$. \qed
\end{corollary}

For example, as polynomials in $\ZZ[t,q,x]$, we have
\begin{align*}
\Sch_{231}^q \cdot \Sch_{231}^q &= q_1(t_2-t_1)\,\Sch_{21}^q + (t_2-t_1)(t_3-t_1)\,\Sch_{231}^q + q_2\,\Sch_{312}^q  \\ 
            & \quad + (t_2-t_1)\,\Sch_{2413}^q + \Sch_{3412}^q,
\end{align*}
where $\Sch_w^q=\Sch_w^q(x,t)$.  To get the corresponding product $\sigma_{231}\circ\sigma_{231}$ in $QH_T^*(\Fl(3))$, simply discard the last two terms.

This provides an easy way to compute equivariant quantum products.  Using Maple\footnote{We thank Anders Buch for providing his code for computing equivariant products, which we adapted to include the quantum case.  The other multiplication tables, as well as code, are available from the authors upon request.}, we computed the full multiplication tables for $QH_T^*(\Fl(n))$ for $n\leq 5$; individual products are also quickly computable for higher $n$.  We include the $n=3$ case in Table~\ref{table_m3}.


%
%

\begin{table}[h]
\[
\begin{array}{|l|l|l|l|} \hline
 u  &  v  & \sigma_u \circ \sigma_v   \\ \hline\hline
213 & 213 & \sigma_{312} + (t_2-t_1)\, \sigma_{213} + q_1\, \sigma_{123} \\ \hline
213 & 132 & \sigma_{231} + \sigma_{312} \\ \hline
213 & 231 & \sigma_{321} + (t_2-t_1)\,\sigma_{231} \\ \hline
213 & 312 & (t_3-t_1)\,\sigma_{312} + q_1\,\sigma_{132} \\ \hline
213 & 321 & (t_3-t_1)\,\sigma_{321} + q_1\,\sigma_{231} + q_1 q_2 \,\sigma_{123} \\ \hline
132 & 132 & \sigma_{231} + (t_3-t_2)\,\sigma_{132} + q_2\,\sigma_{123}   \\ \hline
132 & 231 & (t_3-t_1)\,\sigma_{231} + q_2 \,\sigma_{213} \\ \hline
132 & 312 & \sigma_{321} + (t_3-t_2)\,\sigma_{312} \\ \hline
132 & 321 & (t_3-t_1)\,\sigma_{321} + q_2\,\sigma_{312} + q_1 q_2\,\sigma_{123} \\ \hline
231 & 231 & (t_2-t_1)(t_3-t_1)\,\sigma_{231} + q_2\,\sigma_{312} + q_2(t_2-t_1)\,\sigma_{213} \\ \hline
231 & 312 & (t_3-t_1)\,\sigma_{321} + q_1 q_2\,\sigma_{123} \\ \hline
231 & 321 & (t_2-t_1)(t_3-t_1)\,\sigma_{321} + q_2(t_3-t_1)\,\sigma_{312} \\
    &     &  \quad + q_1 q_2 \,\sigma_{132} + q_1 q_2 (t_2-t_1)\,\sigma_{123} \\ \hline
312 & 312 & (t_2-t_1)(t_3-t_1)\,\sigma_{312} + q_1\,\sigma_{231} + q_1(t_3-t_2)\,\sigma_{132} \\ \hline
312 & 321 & (t_2-t_1)(t_3-t_1)\,\sigma_{321} + q_1(t_3-t_1)\,\sigma_{231} + q_1 q_2\,\sigma_{213}  \\
    &     & \quad + q_1 q_2 (t_3-t_2)\,\sigma_{123} \\ \hline
321 & 321 & (t_2-t_1)(t_3-t_1)(t_3-t_2)\,\sigma_{321} + [q_2(t_3-t_1)(t_3-t_2) + q_1q_2]\,\sigma_{312} \\
    &     & \quad + [q_1(t_2-t_1)(t_3-t_1) + q_1q_2]\,\sigma_{231} + q_1q_2(t_3-t_2)\,\sigma_{132}  \\
    &     & \quad + q_1q_2(t_2-t_1)\,\sigma_{213} + q_1q_2(t_2-t_1)(t_3-t_2)\,\sigma_{123} \\ \hline
\end{array}
\]
\caption{Equivariant quantum products in $QH_T^*Fl(3)$. \label{table_m3}}
\end{table}


\begin{thebibliography}{LaSch}

\bibitem[An]{anderson} Dave~Anderson, ``Positivity in the cohomology of flag bundles (after Graham),'' \texttt{arXiv:0711.0983 [math.AG]}.

\bibitem[AC]{ac} Dave~Anderson and Linda~Chen, ``Positivity of equivariant Gromov-Witten invariants,'' \texttt{arXiv:1110.5900 [math.AG]}.

\bibitem[AGM]{agm} Dave~Anderson, Stephen~Griffeth, and Ezra~Miller, ``Positivity and Kleiman transversality in equivariant $K$-theory of homogeneous spaces,'' {\em J. Eur. Math. Soc.} {\bf 13} (2011), 57--84.

\bibitem[BGG]{bgg} I.N. Bernstein, I.M. Gelfand, S.I. Gelfand, ``Schubert cells and cohomology of the space $G/P$,'' \emph{Russian Math. Surveys} {\bf 28} (1973), 1--26.

\bibitem[Be]{bertram} Aaron~Bertram, ``Quantum Schubert calculus,'' {\em Adv. Math.} {\bf 128} (1997), no. 2, 289--305.

\bibitem[BCS]{bcs} Tom~Braden, Linda~Chen, and Frank~Sottile, ``The equivariant Chow rings of Quot schemes,'' \newblock{\em Pacific J.~Math.} {\bf 238} (2008), no.~2, 201--232.


\bibitem[Ch]{chen} Linda~Chen, ``Quantum cohomology of flag manifolds,'' {\em Adv. Math.} {\bf 174} (2003), no. 1, 1--34.

\bibitem[CF1]{cf-flags} Ionu\c{t}~Ciocan-Fontanine, ``The quantum cohomology ring of flag varieties,'' {\em Trans. Amer. Math. Soc.} {\bf 351} (1999), no. 7, 2695--2729.

\bibitem[CF2]{cf-partial} Ionu\c{t}~Ciocan-Fontanine, ``On quantum cohomology rings of partial flag varieties,'' {\em Duke Math. J.} {\bf 98} (1999), no. 3, 485--524.

\bibitem[D]{d} Michel~Demazure, ``D\'{e}singularization des vari\'{e}t\'{e}s de Schubert g\'{e}n\'{e}ralis\'{e}e,'' \emph{Ann. Sci. Ecole Norm. Sup.} {\bf 7} (1974) 53--88.

\bibitem[EG]{eg} Dan~Edidin and William~Graham, ``Equivariant intersection theory,'' \newblock{\em Invent. Math.} {\bf 131} (1998), 595--634.

\bibitem[FGP]{fgp} Sergey~Fomin,  Sergei~Gelfand, and Alexander~Postnikov, ``Quantum Schubert polynomials,'' {\em J. Amer. Math. Soc.} {\bf 10} (1997), no. 3, 565--596.

\bibitem[Fu]{fulton} William~Fulton, ``Universal Schubert polynomials,'' {\em Duke Math. J.} {\bf 96} (1999), no. 3, 575--594.

\bibitem[FP]{fp} William~Fulton and Rahul~Pandharipande, ``Notes on stable maps and quantum cohomology,'' {\em Algebraic geometry---Santa Cruz 1995}, 45--96, Proc. Sympos. Pure Math., 62, Part 2, Amer. Math. Soc., Providence, RI, 1997.

\bibitem[Kim]{kim-eq} Bumsig~Kim, ``On equivariant quantum cohomology,'' {\em Internat. Math. Res. Notices} {\bf 1996}, no. 17, 841--851.

\bibitem[KiMa]{kima} Anatol~N.~Kirillov and Toshiaki~Maeno, ``Quantum double Schubert polynomials, quantum Schubert polynomials and Vafa-Intriligator formula,'' {\em Discrete Math.} {\bf 217} (2000), no. 1-3, 191--223.

\bibitem[Kl]{kleiman} Steven~Kleiman, ``The transversality of a generic translate,'' {\em Compositio Math.} \textbf{28} (1974), 287--297.

 \bibitem[KnMi]{knmi} Allen~Knutson and Ezra~Miller, ``Gr\"{o}bner geometry of Schubert polynomials,'' {\em Duke Math. J.} {\bf 119} (2003), no. 2, 221-260.

\bibitem[LaSh1]{ls-acta} Thomas~Lam and Mark~Shimozono, ``Quantum cohomology of $G/P$ and homology of affine Grassmannian,'' {\em Acta Math.} {\bf 204} (2010), no. 1, 49--90.

\bibitem[LaSh2]{ls} Thomas~Lam and Mark~Shimozono, ``From quantum Schubert polynomials to $k$-Schur functions via the Toda lattice", \texttt{arxiv:1010.4047}.

\bibitem[LaSh3]{ls2} Thomas~Lam and Mark~Shimozono, ``Quantum double Schubert polynomials represent Schubert classes", \texttt{arxiv:1108.4958v1}.

\bibitem[LaSch]{lsch} Alain~Lascoux and M.-P.~Sch\"{u}tzenberger, ``Polyn\^{o}mes de Schubert,'' \emph{C.R. Acad. Sci. Paris S\'{e}r. I Math} {\bf 294} (1982), 447--450.

\bibitem[LM]{lm} Luc~Lapointe and Jennifer~Morse, ``Quantum cohomology and the $k$-Schur basis,'' \emph{Trans. Amer. Math. Soc.} {\bf 360} (2008), 2021--2040.

\bibitem[Mi1]{mihalcea-eqsc} Leonardo~Mihalcea, ``Equivariant quantum Schubert calculus,'' {\em Adv. Math.} {\bf 203} (2006), no. 1, 1--33.

\bibitem[Mi2]{mihalcea-positivity} Leonardo~Mihalcea, ``Positivity in equivariant quantum Schubert calculus,'' {\em Amer. J. Math.} {\bf 128} (2006), no. 3, 787--803.

\bibitem[Mi3]{mihalcea-giambelli} Leonardo~Mihalcea, ``Giambelli formulae for the equivariant quantum cohomology of the Grassmannian,'' {\em Trans. Amer. Math. Soc.} {\bf 360} (2008), no. 5, 2285--2301.

\bibitem[Sp]{speiser} Robert~Speiser, ``Transversality theorems for families of maps,'' {\em Algebraic geometry (Sundance, UT, 1986)}, 235--252, Lecture Notes in Math., vol. 1311, Springer-Verlag, 1988.


\end{thebibliography}
\end{document}